\documentclass{amsart}

\usepackage{amssymb}
\usepackage{amsmath}
\usepackage{verbatim}
\usepackage[shortlabels]{enumitem}
\usepackage{url}
\usepackage{graphicx}
\usepackage{color}

\usepackage[foot]{amsaddr}

\usepackage{mathrsfs}

\newcommand{\eps}{\varepsilon}
\newcommand{\rmRe}{\operatorname{Re}}
\newcommand{\rmIm}{\operatorname{Im}}
\newcommand{\Bff}{\mathbf}
\newcommand{\BigO}{\mathcal{O}}

\newtheorem{theorem}{Theorem}[section]
\newtheorem{proposition}[theorem]{Proposition}
\newtheorem{corollary}[theorem]{Corollary}
\newtheorem{lemma}[theorem]{Lemma}

\theoremstyle{definition}
\newtheorem{definition}[theorem]{Definition}
\theoremstyle{remark}
\newcommand\xqed[1]{%
  \leavevmode\unskip\penalty9999 \hbox{}\nobreak\hfill
  \quad\hbox{#1}}
\newtheorem{xremark}[theorem]{Remark}
\newenvironment{remark}{\begin{xremark}}{\xqed{$\triangle$}\end{xremark}}
\newtheorem*{acknowledgements}{Acknowledgements}

\numberwithin{equation}{section}

\newcommand{\ii}{\textnormal{i}}
\newcommand{\dd}{\textnormal{d}}
\newcommand{\ee}{\textnormal{e}}

\newcommand{\shift}{\mathcal{S}}
\newcommand{\fourier}{\mathcal{F}}
\newcommand{\barg}{\mathfrak{B}}
\newcommand{\Fock}{\mathfrak{F}}

\newcommand{\arsinh}{\operatorname{arsinh}}
\newcommand{\erf}{\operatorname{erf}}
\newcommand{\JJ}{\|(x + \gamma)^k\|_{\Fock}^2}

\begin{document}

\title[Shifted H.\ O.\ and the hypoelliptic Laplacian]{The shifted harmonic oscillator and the hypoelliptic Laplacian on the circle}

\author{Boris Mityagin}
\address[Boris Mityagin]{Department of Mathematics, The Ohio State University, 231 West 18th Ave,
Columbus, OH 43210, USA}
\email{mityagin.1@osu.edu, boris.mityagin@gmail.com}

\author{Petr Siegl}
\address[Petr Siegl]{School of Mathematics and Physics, Queen's University Belfast, University Road, Belfast, BT7 1NN, UK}
\email{p.siegl@qub.ac.uk}

\author{Joe Viola}
\address[Joe Viola]{Laboratoire de Math\'ematiques J.~Leray, UMR 6629 du CNRS, Universit\'e de Nantes, 2, rue de la Houssini\`ere, 44322 Nantes Cedex 03, France}
\email{Joseph.Viola@univ-nantes.fr}



\begin{abstract}
We study the semigroup generated by the hypoelliptic Laplacian on the circle and the maximal bounded holomorphic extension of this semigroup. Using an orthogonal decomposition into harmonic oscillators with complex shifts, we describe the domain of this extension and we show that boundedness in a half-plane corresponds to absolute convergence of the expansion of the semigroup in eigenfunctions. This relies on a novel integral formula for the spectral projections which also gives asymptotics for Laguerre polynomials in a large-parameter regime.
\end{abstract}

\maketitle

\section{Introduction}

We consider the hypoelliptic Laplacian on the circle \cite[Sec.~1]{Bismut_2008},
\begin{equation}\label{eq:def_HEL}
	L_b = \frac{1}{2}\left(-\frac{\partial^2}{\partial y^2} + y^2 - 1\right) - b y\frac{\partial}{\partial x},
\end{equation}
acting as an unbounded operator on $L^2(\Bbb{T}_x \times \Bbb{R}_y)$. We use the convention that the circle $\Bbb{T}$ is $[0, 2\pi]$ with ends identified, so that $\{\ee^{\ii n x}\}_{n \in \Bbb{Z}}$ forms an orthonormal basis of $L^2(\Bbb{T}_x)$. Furthermore, when
\begin{equation}\label{eq:def_En}
	E_n = \{\ee^{\ii n x}f(y)\::\: f \in L^2(\Bbb{R}_y)\}, \quad n \in \Bbb{Z},
\end{equation}
we have the orthogonal decomposition
\begin{equation}\label{eq:decomp_En}
	L^2(\Bbb{T}_x \times \Bbb{R}_y) = \overline{\bigoplus_{n \in\Bbb{Z}} E_n}.
\end{equation}

For $a \in \Bbb{R}$ we define the shifted harmonic oscillator acting on $L^2(\Bbb{R}_y)$:
\begin{equation}\label{eq:def_SHO}
	P_a = \frac{1}{2}(-\partial_y^2 + (y-\ii a)^2 - 1).
\end{equation}
On the $L_b$-invariant subspaces $\{E_n\}_{n \in \Bbb{Z}}$,
\begin{equation}\label{eq:HEL_on_En}
	L_b|_{E_n} = P_{bn} + \frac{1}{2}(bn)^2.
\end{equation}
The shifted harmonic oscillator is obtained by adding a relatively bounded perturbation of the harmonic oscillator $P_0$, which implies that the resolvent of $P_a$ is compact. This allows us to work with the spectral decomposition of $P_a$ despite the fact that $P_a$ is not a normal operator. (See \cite[Sec.~2.3]{Mityagin_Siegl_Viola_2016}.)

Recent works have significantly advanced our understanding of the poor --- but not too poor --- spectral and pseudospectral properties of the shifted harmonic oscillator. The spectral projection norms grow exponentially rapidly, but in only the square root of the eigenvalue, \cite[Thm.~2.6]{Mityagin_Siegl_Viola_2016}. The resolvent norm grows rapidly in a parabolic region in the complex plane \cite[Sec.~VII.E]{Krejcirik_Siegl_Tater_Viola_2015}. And the semigroup $\ee^{-tP_a}$ may be extended to a bounded operator on $L^2(\Bbb{R})$ for $t$ in the half-plane $\{\rmRe t > 0\}$, but the operator norm becomes extremely large as $\rmRe t \to 0^+$ unless $\rmIm t$ is near $2\pi \Bbb{Z}$, \cite{Viola_2017}.

The connection between the spectral projection norms for the shifted harmonic oscillators and Laguerre polynomials 
\begin{equation}\label{eq:def_laguerre}
	L_k^{(0)}(x) = \sum_{m = 0}^k \frac{1}{m!}\binom{k}{m} (-x)^{2m}
\end{equation}
when $x < 0$ was made in \cite[Eq.~(2.20)]{Mityagin_Siegl_Viola_2016}. This connection came from \cite[Formula 7.374.7]{GraRyz_2000} which relates certain integrals involving Hermite functions to the Laguerre polynomials. Asymptotics for Laguerre polynomials used in \cite{Mityagin_Siegl_Viola_2016} come from \cite[Thm.~8.22.3]{Szego_OP} and were proven in \cite{Perron_1921}.

The sequence $L^{(0)}_k(-1)$ is of combinatorial interest \cite[Seq.\ A002720]{OEIS}, for instance because it is the average number of increasing subsequences of a $k$-long random permutation. We refer the reader to \cite{Lifschitz_Pittel_1981} for a discussion, references, and estimates on the moments of the corresponding random variable.

Any information about the spectral projection norms for the shifted harmonic oscillator leads to information about the Laguerre polynomials (when $x < 0$), and vice versa.

As $|n|$ increases, the addition of $\frac{1}{2}(bn)^2$ in \eqref{eq:HEL_on_En} pushes the spectrum, numerical range, and pseudospectrum of $L_b|_{E_n}$ towards the right in the complex plane and therefore $\ee^{-tL_b}|_{E_n}$ for $\rmRe t > 0$ could be expected to be better-behaved. On the other hand, the pseudospectral properties of $P_{bn}$ worsen.

The goal of the present work is to explore the competition between these two phenomena, and in doing so to describe connections between models in kinetic theory, spectral instability for non-self-adjoint operators, and asymptotics for special functions (specifically, the Laguerre polynomials). In Section \ref{sec:bdd}, we describe the precise shape and other characteristics of the set on which the graph closure of $\ee^{-tL_b}$ from the eigenfunctions \eqref{eq:HEL_eigenfunctions} of $L_b$ is bounded (Figure \ref{fig:Rcontours}). In Section \ref{sec:proj} we consider the spectral decomposition of the hypoelliptic Laplacian $L_b$ using the spectral projection of the shifted harmonic oscillators: we show in Theorem \ref{thm:sigma_tau} that absolute convergence of the spectral representation of $\ee^{-tL_b}$ corresponds to boundedness of $\ee^{-(t_1 + \ii t_2)L_b}$ for all $t_2 \in \Bbb{R}$ simultaneously. In Section \ref{s:integral} we prove a new integral formula for these spectral projection norms using the Bargmann transform. In Section \ref{s:laplace} we describe how Laplace's method applied to this integral formula gives sharp asymptotics which allow us to prove Theorem \ref{thm:sigma_tau}; details of the long and somewhat technical proof are postponed to Sections \ref{s:laplace_theta}--\ref{s:laplace_proofs}. In Section \ref{s:laguerre} we relate these formulas to the Laguerre polynomials.

\begin{remark}
We use $\Bbb{N}$ to denote all non-negative integers (including zero). Norms, unless otherwise specified, are for functions the $L^2$ norm and for operators the operator norm induced by the $L^2$ norm. We use the notation $\BigO(b)$ to denote a quantity bounded in absolute value by $C|b|$ for some $C > 0$. A subscript, such as $\BigO_a(b)$, indicates that the constant $C$ depends on the parameter $a$. An index of notation can be found in Appendix \ref{app:symbol_index} at the end of the work.
\end{remark}

\begin{acknowledgements}The authors would like to thank the Institut Henri Poincar\'e for very pleasant working conditions during their stay in September 2018. The authors would also like to thank D.\ Zeilberger for helpful discussions. The third author gratefully acknowledges the support of the R\'egion Pays de la Loire through the project EONE (\'Evolution des Op\'erateurs
Non-Elliptiques).
\end{acknowledgements}

\begin{figure}
\includegraphics[width = .2\textwidth]{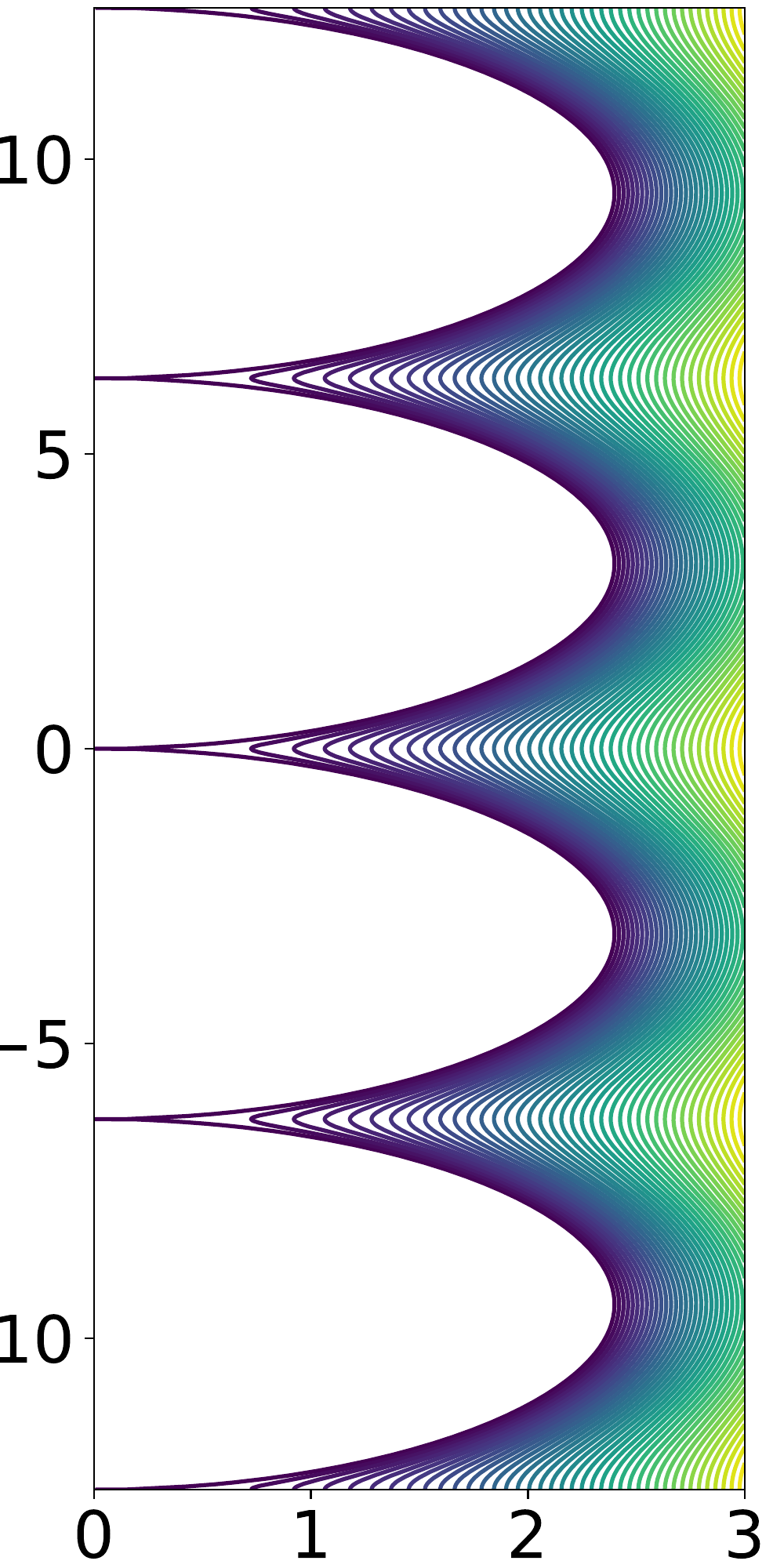}
\includegraphics[width = .6\textwidth]{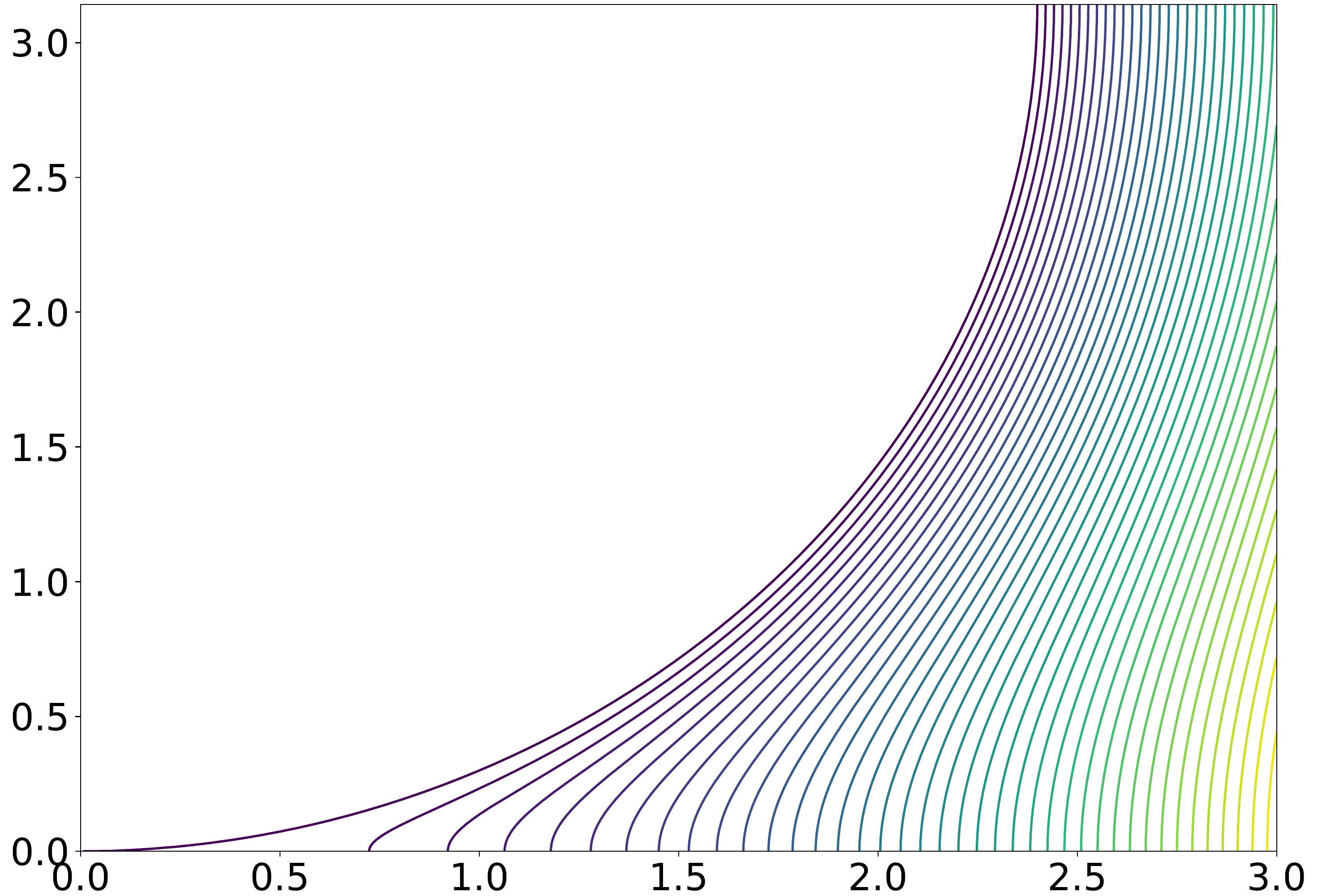}
\caption{Function $R$ from \eqref{eq:bdd_cond} governing boundedness and return to equilibrium for $\ee^{-t L_b}$ for $\rmRe t \in [0, 3]$ while $\rmIm t$ in $[-4\pi, 4\pi]$ (left) or $[0, \pi]$ (right); contours $\{0, 0.015, 0.03, \dots, 0.6\}$ begin where $\ee^{-tL_b}$ is bounded. See Theorem \ref{thm:bdd}.}
\label{fig:Rcontours}
\end{figure}

\section{Boundedness and return to equilibrium}\label{sec:bdd}

The hypoelliptic Laplacian on the circle decomposes into shifted harmonic oscillators, as described in \eqref{eq:HEL_on_En}. With results on shifted harmonic oscillators, we have an exact description of the evolution $\ee^{-tL_b}$ acting on $L^2(\Bbb{T} \times \Bbb{R})$ viewed as the graph closure \cite[Prop.~2.1, 2.23]{Aleman_Viola_2018} from the dense set of eigenfunctions
\begin{equation}\label{eq:HEL_eigenfunctions}
	f_{bn, k}(x, y) = (2\pi)^{-1}\ee^{\ii bn x}h_k(y-\ii bn)
	\in \ker\left(L_b - (\frac{1}{2}(bn)^2 + k)\right), \quad n \in \Bbb{Z}, k \in \Bbb{N}.
\end{equation}
Here, 
\begin{equation}\label{eq:def_hermite}
	h_k(y) = \frac{1}{\sqrt{2^k k! \sqrt{\pi}}}\left(y - \frac{\dd}{\dd y}\right)^k \ee^{-y^2/2} \in \ker(P_0 - k)
\end{equation}
are the Hermite functions. On $\operatorname{Span}\{f_{bn, k}\}_{n \in \Bbb{Z}, k \in \Bbb{N}}$ we can define
\begin{equation}\label{eq:HEL_eigenfunctions_semigroup}
	\ee^{-tL_b}\left(\sum \alpha_{n, k}f_{bn, k}\right) = \sum\ee^{-t(\frac{1}{2}(bn)^2 + k)}\alpha_{n, k}f_{bn, k},
\end{equation}
for $\{\alpha_{n, k}\}_{n \in \Bbb{Z}, k \in \Bbb{N}}$ a finitely non-zero sequence of complex numbers, but convergence for infinite sums is far from guaranteed.

We now identify the subset of $t \in \Bbb{C}$ for which $\ee^{-tL_b}$ extends to a bounded operator. Somewhat surprisingly, this set is independent of $b\in \Bbb{R}\backslash\{0\}$.

\begin{theorem}\label{thm:bdd}
Let $b \in \Bbb{R} \backslash \{0\}$. Recall $L_b$ from \eqref{eq:def_HEL} acting on $L^2(\Bbb{T}_x \times \Bbb{R}_y)$. For $t \in \Bbb{C}$, the operator $\ee^{-tL_b}$, viewed as the graph closure starting from the eigenfunctions \eqref{eq:HEL_eigenfunctions}, is bounded if and only if $t\in 2\pi\ii\Bbb{Z}$ or if $\rmRe t > 0$ and
\begin{equation}\label{eq:bdd_cond}
	R(t) = \frac{1}{2}\rmRe t - \frac{\cosh (\rmRe t) - \cos (\rmIm t)}{\sinh (\rmRe t)} \geq 0.
\end{equation}
\end{theorem}

\begin{proof}
When $t \in 2\pi\ii \Bbb{Z}$, it is enough to observe that $\ee^{-t k} = 1$ for all $k \in \Bbb{N}$ and so $\ee^{-tL_b}$ acts on each $E_n$ in \eqref{eq:def_En} as multiplication by $\ee^{-\frac{1}{2}t(bn)^2}$. Since the spaces $\{E_n\}_{n \in \Bbb{N}}$ are mutually orthogonal, $\ee^{-tL_b}$ is in fact unitary when $t \in 2\pi\ii \Bbb{Z}$.

Writing $(t_1, t_2) = (\rmRe t, \rmIm t)$, when $t_1 > 0$, we can use the exact formula for the norm of $\ee^{-tP_a}$ in \cite[Eq.~(1.1), Ex.~5.4]{Viola_2017} which we re-prove below in Proposition \ref{prop:SHO_norm}. With this formula and \eqref{eq:HEL_on_En}, we obtain
\begin{equation}\label{eq:bdd_proof}
	\begin{aligned}
	\|\ee^{-tL_b}|_{E_n}\| &= \ee^{-\frac{1}{2}t_1 (bn)^2}\|\ee^{-tP_{bn}}\|
	\\ &= \exp\left(-\frac{1}{2}t_1 (bn)^2 + (bn)^2\frac{\cosh t_1 - \cos t_2}{\sinh t_1}\right)
	\\ &=\exp\left((bn)^2\left(-\frac{1}{2}t_1 + \frac{\cosh t_1 - \cos t_2}{\sinh t_1}\right)\right).
	\end{aligned}
\end{equation}
If \eqref{eq:bdd_cond} fails, then the norm of $\ee^{-tL_b}|_{E_n}$ blows up as $|n| \to \infty$. If \eqref{eq:bdd_cond} holds, then since the spaces $\{E_n\}_{n \in \Bbb{Z}}$ are mutually orthogonal,
\begin{equation}\label{eq:norm_one}
	\|\ee^{-tL_b}\| = \sup_{n \in \Bbb{Z}} \|\ee^{-tL_b}|_{E_n}\| = 1.
\end{equation}
\end{proof}

\begin{remark}\label{rem:bdd_SHO_norm1}
Let us note that, if $\rmRe t > 0$,
\begin{equation}
	\|\ee^{-tL_b}|_{E_n}\| = \|\ee^{-t(bn)^2(P_1 + \frac{1}{2})}\|
\end{equation}
and
\begin{equation}
	R(t) = -\log\|\ee^{-t(P_1 + \frac{1}{2})}\|.
\end{equation}
The boundedness of the evolution of the hypoelliptic Laplacian on the circle therefore reduces to whether the norm of the evolution of a model shifted harmonic oscillator (this time with spectrum $\Bbb{N}+\{\frac{1}{2}\}$) is less than one.

From \eqref{eq:norm_one} we also note that either $\|\ee^{-tL_b}\| = \infty$ or $\|\ee^{-tL_b}\| = 1$.
\end{remark}

To analyze the set where $\ee^{-tL_b}$ is bounded, we consider the function
\begin{equation}\label{eq:def_Fbdry}
	F(t) = 1-(\cosh t - \frac{1}{2}t\sinh t)
\end{equation}
Theorem \ref{thm:bdd} states that $\ee^{-tL_b}$ is bounded if and only if $\rmRe t \geq 0$ and
\begin{equation}\label{eq:Fbdry_bdd_cond}
	F(\rmRe t) \geq 1-\cos (\rmIm t).
\end{equation}
Let us also note that
\begin{equation}
	F'(t) = \frac{1}{2}(t\cosh t - \sinh t),
\end{equation}
and
\begin{equation}
	F''(t) = \frac{1}{2}t \sinh t = \sum_{k = 0}^\infty \frac{(k+1)t^{2k+2}}{(2k+2)!}.
\end{equation}
We can immediately deduce the series expansion
\begin{equation}\label{eq:Fbdry_Taylor}
	F(t) = \sum_{k = 0}^\infty \frac{(k+1)t^{2k+4}}{(2k+4)!},
\end{equation}
as well as the facts that $F(t)$ is even, nonnegative, vanishes to fourth order at $t = 0$, and increases to $+\infty$ as $t \to \infty$.

In particular, we note that there exists a unique $\tau \geq 0$ such that
\begin{equation}\label{eq:def_t1star}
	F(\tau) = 2.
\end{equation}
If we observe that
\begin{equation}
	2-F(t) = 2\left(\cosh\frac{t}{2}\right)\left(\cosh\frac{t}{2} - \frac{t}{2}\sinh\frac{t}{2}\right),
\end{equation}
we see that $\frac{1}{2}\tau$ is the unique positive fixed point of hyperbolic cotangent:
\begin{equation}\label{eq:t1star_coth}
	\frac{\tau}{2} = \coth \frac{\tau}{2}, \quad \tau > 0.
\end{equation}
Furthermore, when $\rmRe t \geq 0$, we can solve \eqref{eq:Fbdry_bdd_cond} for equality in terms of $\rmIm t$ if and only if $\rmRe t \leq \tau$. Because of the central role of $\tau$ in this work, we record its definition.

\begin{definition}\label{def:tau}
Let $\tau$ be the unique positive solution to
\begin{equation}
	\coth \frac{\tau}{2} = \frac{\tau}{2}.
\end{equation}
We note that $\tau \approx 2.39926$.
\end{definition}

We now collect some information on the boundary of the set of $t \in \Bbb{C}$ where $\ee^{-tL_b}$ is bounded, which is illustrated in Figure \ref{fig:Rcontours}.

\begin{proposition}\label{prop:bdd_shape}
Let $\tau$ be as in Definition \ref{def:tau} and $F$ as in \eqref{eq:def_Fbdry}. The curve
\begin{equation}
	\{(x, y) \in [0, \tau] \times [0, \pi] \::\: 1-\cos y = F(x)\}
\end{equation}
represents a convex function $y(x)$ where
\begin{equation}\label{eq:Fbdry_vert}
	\lim_{x \to \tau^-} y'(x) = +\infty
\end{equation}
and, for some $C \geq 0$,
\begin{equation}\label{eq:Fbdry_order4}
	\frac{1}{2\sqrt{3}} x^2 \leq y \leq \frac{1}{2\sqrt{3}} x^2 + C x^4, \quad \forall x \in [0, \tau].
\end{equation}
\end{proposition}

\begin{proof}
For $x \in (0, \tau)$, where $\sin y \in (0, 1]$, we observe from $1-\cos y = F(x)$ that
\begin{equation}\label{eq:shape_dy}
	y' = \frac{F'(x)}{\sin y}.
\end{equation}
As $x \to \tau^-$, $F(x) \to 2^-$ and therefore $\sin y \to 0^+$. When $x > 0$, $F'(x)$ is positive and increasing since $F'(0) = 0$ and $F''(x) = \frac{1}{2}x\sinh x > 0$. Putting these two facts into \eqref{eq:shape_dy} proves \eqref{eq:Fbdry_vert}. As for the second derivative of $y$,
\begin{equation}\label{eq:Fbdry_d2}
	\begin{aligned}
	y'' &= \frac{1}{\sin^3 y}(F''(x) \sin^2 y - F'(x)^2\cos y)
	\\ &= \frac{1}{\sin^3 y}(F''(x)F(x)(2-F(x)) - F'(x)^2 (F(x) - 1))
	\\ &= \frac{1}{\sin^3 y}(2F''(x) F(x) - F'(x)^2 + F(x)(F'(x)^2  - F''(x)F(x)).
	\end{aligned}
\end{equation}

We will show that $y''(x)$ is positive for $x \in (0, \tau)$ by showing that $y''(x) \sin^3 x$ is positive. The arrangement of terms above is convenient because 
\begin{equation}\label{eq:Fbdry_d2square}
	4(F'(x)^2 - F''(x)F(x)) = (\sinh x - x)^2
\end{equation}
is positive. Furthermore,
\begin{equation}
	4(2F''(x)F(x) - F'(x)^2) = 4F''(x)F(x) - (\sinh x - x)^2
\end{equation}
can be expanded as an infinite series. With the series for  $F''(x) = \frac{1}{2}x\sinh x$ and for $F(x)$ given in \eqref{eq:Fbdry_Taylor}, 
\begin{equation}
	4F''(x)F(x) = x^6 \left(\sum_{j=0}^\infty \frac{2x^{2j}}{(2j+1)!}\right)\left(\sum_{k=0}^\infty \frac{(k+1)x^{2k}}{(2k+4)!}\right).
\end{equation}
For any $\lambda \in \Bbb{R}$, we expand
\begin{equation}
	\begin{aligned}
	4(\lambda F''(x)F(x) &+ F''(x)F(x) - F'(x)^2) = 4\lambda F''(x) F(x) - (\sinh x - x)^2
	\\ &= x^6\lambda\left(\sum_{j=0}^\infty \frac{2x^{2j}}{(2j+1)!}\right)\left(\sum_{k=0}^\infty \frac{(k+1)x^{2k}}{(2k+4)!}\right) - x^6\left(\sum_{j=0}^\infty \frac{x^{2j}}{(2j+3)!}\right)^2
	\\ &= x^6 \sum_{n = 0}\mathop{\sum_{j, k \geq 0}}_{j+k = n} x^{2n}\left(\frac{2\lambda(k+1)}{(2j+1)!(2k+4)!} - \frac{1}{(2j+3)!(2k+3)!}\right)
	\\ &= x^6\sum_{n = 0}\mathop{\sum_{j, k \geq 0}}_{j+k = n} \frac{x^{2n}}{(2j+1)!(2k+3)!}\left(\lambda\frac{k+1}{k+2} - \frac{1}{(2j+2)(2j+3)}\right)
	\end{aligned}
\end{equation}
For $j, k \geq 0$,
\begin{equation}
	\lambda \frac{k+1}{k+2} - \frac{1}{(2j+2)(2j+3)} \geq \frac{\lambda}{2} - \frac{1}{6}.
\end{equation}
We conclude that, for any $\lambda \geq \frac{1}{3}$,
\begin{equation}
	\lambda F''(x)F(x) + F''(x)F(x) - F'(x)^2 \geq 0, \quad \forall x \geq 0.
\end{equation}
In particular, with $\lambda = 1$,
\begin{equation}
	2F''(x)F(x) - F'(x)^2 \geq 0, \quad \forall x \geq 0.
\end{equation}

Along with \eqref{eq:Fbdry_d2} and \eqref{eq:Fbdry_d2square}, we have proven that
\begin{equation}\label{eq:boundary_convexity}
	y''(x) > 0, \quad \forall x \in (0, \tau).
\end{equation}

The approximation of $y(x)$ near $x = 0$ comes from expanding $1-\cos y = F(x)$ to obtain
\begin{equation}
	\frac{y^2}{2} + \BigO(y^4) = \frac{x^4}{24} + \BigO(x^6),
\end{equation}
which gives
\begin{equation}
	y(x) = \frac{x^2}{2\sqrt{3}} + \BigO(x^4).
\end{equation}
The fact that $y(x) \geq \frac{x^2}{2\sqrt{3}}$ comes from the convexity shown in \eqref{eq:boundary_convexity}, which completes the proof of the proposition.
\end{proof}

The shape of the set, near zero, of $t$ for which $\ee^{-tL_b}$ is bounded is similar to that of other hypoelliptic operators like the quadratic Kramers-Fokker-Planck model \cite{Aleman_Viola_2015, Aleman_Viola_2018}. Another behavior characteristic of hypoelliptic operators is the slow return to equilibrium in small times. 

The equilibrium is the eigenfunction associated to the zero eigenvalue. This eigenfunction is, when normalized, the Gaussian $f_{0, 0}(x, y) = (2\pi)^{-1}\ee^{-y^2}$. Recalling that $L_b|_{E_0} = P_0$, the self-adjoint quantum harmonic oscillator with spectrum $\Bbb{N}$, the spectral projection onto $E_0 = \operatorname{Span}\{f_{0, 0}\}$ is the orthogonal projection given by the $L^2$-inner product. As a linear operator on $E_0$,
\begin{equation}
	\|\ee^{-tL_b} - \langle \cdot, f_{0, 0}\rangle f_{0, 0}\|_{\mathcal{L}(E_0)} = \ee^{-t_1}.
\end{equation}
On any other $E_n$, $n \neq 0$, projection onto $f_{0, 0}$ acts as the zero operator. By \eqref{eq:bdd_proof}, with $R(t)$ defined in \eqref{eq:bdd_cond}, whenever $R(t) \geq 0$
\begin{equation}
	\|\ee^{-tL_b} - \langle \cdot, f_{0, 0}\rangle f_{0, 0}\|_{\mathcal{L}(E_n)} = \|\ee^{-tL_b}\|_{\mathcal{L}(E_n)} = \exp(-(bn)^2 R(t)), \quad n \in \Bbb{Z}\backslash \{0\}.
\end{equation}
When $R(t) \geq 0$ the maximum over $n \in \Bbb{Z}\backslash\{0\}$ is achieved for $n = \pm 1$. Taking the maximum over $n \in \Bbb{Z}$, we obtain the norm governing return to equilibrium for this model. In the case $t \geq 0$ real, this has been shown in \cite[Thm.~1]{Gadat_Miclo_2013}.

\begin{corollary}
With $f_{0, 0}(x, y) = (2\pi)^{-1}\ee^{-y^2/2}$ from \eqref{eq:HEL_eigenfunctions} and $R(t)$ from \eqref{eq:bdd_cond}, whenever $\rmRe t > 0$ and $R(t) \geq 0$,
\begin{equation}
	\log \|\ee^{-tL_b} - \langle \cdot, f_{0, 0}\rangle f_{0, 0}\|_{\mathcal{L}(L^2(\Bbb{T}\times\Bbb{R}))} = -\min\{\rmRe t, b^2 R(t)\}.
\end{equation}
\end{corollary}

\section{Comparison with spectral projection norms}\label{sec:proj}

The spectral projections associated with non-self-adjoint operators are often of limited utility because the norms of the spectral projections grow rapidly. Suppose that an operator $A$ with compact resolvent has discrete eigenvalues $\{\lambda_k\}_{k \in \Bbb{N}}$ associated to rank-one spectral projections $\{\Pi_k\}_{k\in\Bbb{N}}$. We say that the spectral decomposition of $\ee^{-tA}$ converges absolutely for some $t \in \Bbb{C}$ if
\begin{equation}
	\sum_{k \in \Bbb{N}}\ee^{-\rmRe(t\lambda_k)}\|\Pi_k\| < \infty.
\end{equation}

One often defines $\ee^{-tA}$ as a one-parameter semigroup (e.g.,\ \cite[Ch.~6]{Davies_Book_2007}) where $u_t = \ee^{-tA}u_0$ solves
\begin{equation}\label{eq:one-param_semigroup}
	\frac{\dd}{\dd t} u_t = -Au_t
\end{equation}
for $u_0$ in an appropriate domain. Whenever $u_0$ is a linear combination of eigenfunctions of $A$,
\begin{equation}
	u_t = \sum_{k \in \Bbb{N}} \ee^{-t\lambda_k}\Pi_k u_0
\end{equation}
solves \eqref{eq:one-param_semigroup} for all $t \in \Bbb{C}$. Furthermore, if $t\in\Bbb{C}$ is such that the spectral decomposition of $\ee^{-tA}$ converges absolutely, $u_0 \mapsto u_t$ extends to a bounded operator. We therefore regard the spectral decomposition as the only possible definition of $\ee^{-tA}$ so long as it converges absolutely.

For the shifted harmonic oscillator $P_a$ from \eqref{eq:def_SHO} for $a \in \Bbb{R}$ and $\rmRe t > 0$, the spectral decomposition of $\ee^{-tP_a}$ converges absolutely despite rapidly-growing spectral projection norms. The spectral projection $\Pi_{a, k}$ of $P_a$ associated to the eigenvalue $k \in \Bbb{N}$ admits the explicit expression
\begin{equation}\label{eq:def_proj_hermite}
	\Pi_{a, k}f(y) = \langle f(y), h_k(y + \ii a)\rangle h_k(y - \ii a)
\end{equation}
for $\{h_k\}_{k\in\Bbb{N}}$ the Hermite functions \eqref{eq:def_hermite}. We know from \cite[Thm.~2.6]{Mityagin_Siegl_Viola_2016}, for $a \in \Bbb{R} \backslash \{0\}$ and as $k \to \infty$,
\begin{equation}\label{eq:MiSiVi}
	\|\Pi_{a, k}\| = \frac{1}{2(2k)^{1/4}\sqrt{|a|\pi}}\exp\left(2^{3/2}|a|\sqrt{k}\right)(1+\BigO_a(k^{-1/2})), \quad k \to +\infty.
\end{equation}
Here $\BigO_a(k^{-1/2})$ denotes a quantity that is bounded by $C_ak^{-1/2}$ for $k \geq 1$ for a constant $C_a > 0$ depending only on $a$. We recall for later reference that this formula comes from \cite[Eq.~(2.20)]{Mityagin_Siegl_Viola_2016}
\begin{equation}\label{eq:MiSiVi_laguerre}
	\|\Pi_{a, k}\| = \ee^{a^2}L^{(0)}_k(-2a^2)
\end{equation}
for $L^{(0)}_k$ the Laguerre polynomials \cite[Ch.~V]{Szego_OP} recalled in \eqref{eq:def_laguerre}, for which there are asymptotics as $k \to \infty$.

In view of \eqref{eq:MiSiVi}, for $\rmRe t > 0$,
\begin{equation}
	f \mapsto \sum_{k\in\Bbb{N}} \ee^{-tk}\Pi_{a, k}f
\end{equation}
is a bounded operator on $L^2(\Bbb{R})$ which agrees with $\ee^{-tP_a}$ on any linear combination of eigenfunctions. This operator is therefore in fact $\ee^{-tP_a}$ which is therefore bounded.

Absolute convergence of the spectral decomposition for the semigroup is more delicate when one considers the rotated harmonic oscillator (the Davies operator, \cite{Davies_1999})
\begin{equation}
	Q_\theta = \frac{1}{2}\left(\ee^{\ii\theta} x^2 - \ee^{-\ii\theta}\frac{\dd^2}{\dd x^2}\right), \quad \theta \in (-\frac{\pi}{2}, \frac{\pi}{2})
\end{equation}
acting on $L^2(\Bbb{R})$. The norms of the associated spectral projections grow exponentially rapidly \cite{Davies_2000, Davies_Kuijlaars_2004, Bagarello_2010, Viola_2013, Henry_2014} in the real part of the eigenvalue $k + \frac{1}{2}, k \in \Bbb{N}$ (whereas the growth of projection norms for the shifted harmonic oscillator is only exponential in $\sqrt{k}$). Therefore the spectral decomposition of $\ee^{-tQ_\theta}$ only converges absolutely on some half-plane $\{\rmRe t > t_\theta^*\}$. It so happens \cite[Eq.~(5.6)]{Viola_2017} that $t_\theta^*$ is the critical value of $\rmRe t$ determining whether the evolution $\ee^{-tQ_{\theta}}$ is bounded:
\begin{equation}
	\sup_{t_2 \in \Bbb{R}} \|\ee^{-(t_1 + \ii t_2)Q_\theta}\|_{\mathcal{L}(L^2(\Bbb{R}))} < \infty \iff t_1 \geq t_\theta^*.
\end{equation}

In light of Proposition \ref{prop:bdd_shape}, we consider the question of whether the same phenomenon occurs for the hypoelliptic Laplacian. The remainder of this work is devoted to showing that indeed, boundedness of $\ee^{-(t+\ii s)L_b}$ for every $s \in \Bbb{R}$ corresponds to absolute convergence of the spectral decomposition of $\ee^{-tL_b}$ (Theorem \ref{thm:sigma_tau}).

Let $\Pi_{a, k}$ from \eqref{eq:def_proj_hermite} be the spectral projection associated to the operator $P_a$ in \eqref{eq:def_SHO} at the eigenvalue $k \in \Bbb{N}$. Using the orthogonal decomposition \eqref{eq:decomp_En} and \eqref{eq:HEL_on_En}, we have the upper bound (when $t = t_1 + \ii t_2$ with $t_1, t_2 \in \Bbb{R}$)
\begin{equation}\label{eq:bdd_from_proj}
	\begin{aligned}
	\|\ee^{-tL_b}\| &\leq \sup_{n \in \Bbb{Z}} \ee^{-\frac{1}{2}t_1 (bn)^2}\|\ee^{-tP_{bn}}\|
	\\ &\leq \sup_{n \in \Bbb{Z}} \sum_{k \in \Bbb{N}} \ee^{-\frac{1}{2}t_1 (bn)^2 - t_1 k}\|\Pi_{bn, k}\|.
	\end{aligned}
\end{equation}
We define $\sigma(b)$ as the infimum of $t_1 > 0$ such that the spectral decomposition for $\|\ee^{-tL_b}\|$ converges absolutely.

\begin{definition}\label{def:sigma}
For $\Pi_{bn, k}$ in \eqref{eq:def_proj_hermite}, let
\begin{equation}\label{eq:def_sigma}
	\sigma(b) = \inf\{t > 0 \::\: \sum_{n \in \Bbb{Z}} \sum_{k \in \Bbb{N}}\ee^{-t(\frac{1}{2}(bn)^2 + k)}\|\Pi_{bn, k}\| < \infty\}.
\end{equation}
\end{definition}

We begin by using \cite[Thm.~2.6]{Mityagin_Siegl_Viola_2016} to give bounds for the sums in Definition \ref{def:sigma} if either $n \in \Bbb{Z}$ or $k \in \Bbb{N}$ is fixed.

\begin{proposition}\label{prop:sigma_nk_fixed}
Let $b \in \Bbb{R}\backslash\{0\}$ and let $\Pi_{bn, k}$ be as in \eqref{eq:def_proj_hermite}. If $n \in \Bbb{Z}$ is fixed, then for any $t > 0$,
\begin{equation}
	\sum_{k\in\Bbb{N}} \ee^{-t(\frac{1}{2}(bn)^2 + k)}\|\Pi_{bn, k}\| < \infty.
\end{equation}
If, on the other hand, $k \in \Bbb{N}$ is fixed, then
\begin{equation}
	\sum_{n \in \Bbb{Z}} \ee^{-t(\frac{1}{2}(bn)^2 + k)}\|\Pi_{bn, k}\| \begin{cases} = \infty, & t \leq 2, \\ < \infty, & t > 2,\end{cases}
\end{equation}
and accordingly
\begin{equation}\label{eq:sigma_geq_2}
	\sigma(b) \geq 2
\end{equation}
for $\sigma$ in Definition \ref{def:sigma}.
\end{proposition}

\begin{proof}
For $n$ fixed and $t > 0$, by \eqref{eq:MiSiVi} as $k \to \infty$
\begin{equation}
	\ee^{-t(\frac{1}{2}(bn)^2 + k)}\|\Pi_{bn, k}\| = \frac{\ee^{-t(bn)^2/2}}{2(2k)^{1/4}\sqrt{\pi|bn|}}\exp\left(-tk\left(1 - \frac{2^{3/2}|bn|}{\sqrt{k}}\right)\right)(1+\BigO_{bn}(k^{-1/2})),
\end{equation}
which has a finite sum over $k\in\Bbb{N}$.

Using \eqref{eq:MiSiVi_laguerre},
\begin{equation}
	\ee^{-t(\frac{1}{2}(bn)^2 + k)}\|\Pi_{bn, k}\| = \ee^{-tk + \frac{1}{2}(2-t)(bn)^2}L^{(0)}_k(-2(bn)^2).
\end{equation}
Since $L^{(0)}_k(-2(bn)^2)$ is a polynomial in $n \in \Bbb{Z}$, the sum converges if and only if $t > 2$.
\end{proof}

With Proposition \ref{prop:bdd_shape}, we know that for certain $t \in \Bbb{C}$, the expansion for $\ee^{-tL_b}$ cannot converge absolutely for all $f \in L^2(\Bbb{R})$ because the operator $\ee^{-tL_b}$ is not bounded. We therefore have the following lower bound for $\sigma$.

\begin{corollary}\label{cor:sigma_lb}
For $\tau$ in Definition \ref{def:tau} and $\sigma(b)$ in Definition \ref{def:sigma}, for any $b \neq 0$,
\begin{equation}
	\sigma(b) \geq \tau.
\end{equation}
\end{corollary}

\begin{remark}
We can now see that the implicit constant in \eqref{eq:MiSiVi} cannot be universal and must depend on the parameter $a$. Suppose that for some $b \in \Bbb{R}\backslash \{0\}$ there were some $C > 0$ such that, for all $n \in \Bbb{Z}\backslash \{0\}$ and $k \in \Bbb{N}\backslash \{0\}$,
\begin{equation}\label{eq:MiSiVi_universal_contradiction}
	\|\Pi_{bn, k}\| \leq \frac{C}{2(2k)^{1/4}\sqrt{\pi|bn|}}\exp(2^{3/2}|bn|\sqrt{k}).
\end{equation}
Since
\begin{equation}
	-2(\frac{1}{2}(bn)^2 + k) + 2^{3/2}|bn|\sqrt{k} = -2\left(2^{-1/2}|bn| + \sqrt{k}\right)^2 \leq 0,
\end{equation}
for any $t > 2$,
\begin{equation}
	\sum_{k > 0}\sum_{n \in \Bbb{Z}\backslash\{0\}} \ee^{-t(\frac{1}{2}(bn)^2 + k)}\|\Pi_{bn, k}\| \leq \sum_{k > 0}\sum_{n \in \Bbb{Z}\backslash\{0\}} \frac{C}{2^{5/4}\sqrt{\pi|b|}}\ee^{-(t-2)(\frac{1}{2}(bn)^2 + k)} < \infty.
\end{equation}
Since $\|\Pi_{0, k}\| = 1$ (since the projection is orthogonal) and $\|\Pi_{bn, 0}\| = \ee^{(bn)^2}$ by \eqref{eq:MiSiVi_laguerre}, we would have
\begin{equation}
	\sigma(b) = 2 < 2.3 < \tau \leq \sigma(b).
\end{equation}
But this is not the case and therefore \eqref{eq:MiSiVi_universal_contradiction} cannot hold with a constant independent of $n$.
\end{remark}

\begin{proposition}\label{prop:sigma_log}
With $\sigma(b)$ defined in Definition \ref{def:sigma} and $k_1 = k+\frac{1}{2}$,
\begin{equation}
	\sigma(b) = \limsup_{\min\{|n|, k\}\to \infty} \frac{\log\|\Pi_{bn, k}\|}{\frac{1}{2}(bn)^2 + k_1}.
\end{equation}
\end{proposition}

\begin{proof}
Let
\begin{equation}
	S(b) = \limsup_{\min\{|n|, k\}\to \infty} s(b, n, k),
\end{equation}
where (using \eqref{eq:MiSiVi_laguerre})
\begin{equation}
	s(b, n, k) = \frac{\log\|\Pi_{bn, k}\|}{\frac{1}{2}(bn)^2 + k_1} = \frac{(bn)^2 + \log L^{(0)}_k(-2(bn)^2)}{\frac{1}{2}(bn)^2 + k_1}.
\end{equation}
Because $L^{(0)}_k$ is a polynomial, this quantity tends to $2$ as $n \to \infty$ for any fixed $k$. Therefore
\begin{equation}
	S(b) \geq 2.
\end{equation}

If $t > S(b) \geq 2$ and $t' \in (S(b), t)$, then there exists some $N > 0$ such that
\begin{equation}
	|n|, k \geq N \implies s(b, n, k) \leq t'.
\end{equation}
This in turn implies that
\begin{equation}
	|n|, k \geq N \implies \|\Pi_{bn, k}\| \leq \ee^{t'(\frac{1}{2}(bn)^2 + k_1)}.
\end{equation}
Therefore
\begin{equation}
	\sum_{k, |n| \geq N} \ee^{-t(\frac{1}{2}(bn)^2 + k)}\|\Pi_{bn, k}\| \leq \sum_{k, |n| \geq N} \ee^{-(t-t')(\frac{1}{2}(bn)^2 + k) + \frac{1}{2}t'} < \infty.
\end{equation}
Moreover, by Proposition \ref{prop:sigma_nk_fixed}, since $t > 2$
\begin{equation}
	\sum_{\min\{k, |n|\} < N} \ee^{-t(\frac{1}{2}(bn)^2 + k)}\|\Pi_{bn, k}\| < \infty
\end{equation}
as well. Therefore the sum over all $n \in \Bbb{Z}$ and $k\in \Bbb{N}$ converges, and we have shown that
\begin{equation}
	\sigma(b) \leq S(b).
\end{equation}

Choose a sequence $\{(n_j, k_j)\}_{j \in \Bbb{N}}$ in $\Bbb{Z}\times \Bbb{N}$ such that $|n_j|, k_j \to \infty$ as $j \to \infty$ and $s(b, n_j, k_j) \to S(b)$. Therefore, for any $t < S(b)$ and $t' \in (t, S(b))$, $s(b, n_j, k_j) \geq t'$ for sufficiently large $j$ which means
\begin{equation}
	\|\Pi_{bn_j, k_j}\|\geq \ee^{t'(\frac{1}{2}(bn_j)^2 + k_j+ \frac{1}{2})}.
\end{equation}
Then among the terms in the sum \eqref{eq:def_sigma} we have
\begin{equation}
	\ee^{-t(\frac{1}{2}(bn_j)^2 + k_j)}\|\Pi_{bn_j, k_j}\| \geq \ee^{(t'-t)(\frac{1}{2}(bn_j)^2 + k_j) + \frac{1}{2}t'} \to \infty, \quad j \to \infty.
\end{equation}
The sum therefore diverges, and we have proven that 
\begin{equation}
	\sigma(b) \geq S(b)
\end{equation}
which completes the proof of the proposition.
\end{proof}

The goal of what follows is to improve our understanding of the spectral projections norms of the shifted harmonic oscillator enough to show that $\sigma(b)$ and $\tau$ are equal. The proof of this theorem can be found in Section \ref{ss:proof_sigma_tau}, though the estimates on the spectral projections depend on an involved application of Laplace's method summarized in Section \ref{ss:laplace_strategy} and carried out in Sections \ref{s:laplace_theta}--\ref{s:laplace_proofs}.

\begin{theorem}\label{thm:sigma_tau}
For $\tau$ in Definition \ref{def:tau} and $\sigma(b)$ in Definition \ref{def:sigma}, for any $b \in \Bbb{R}\backslash\{0\}$,
\begin{equation}
	\sigma(b) = \tau.
\end{equation}
\end{theorem}

\begin{remark}
In the proof of Theorem \ref{thm:sigma_tau} we see that the limit superior in Proposition \ref{prop:sigma_log} is the limit as $n \to \infty$ whenever $k = k(n)$ and
\begin{equation}
	 \lim_{n \to \infty}\frac{k(n)}{(bn)^2} = \frac{1}{2\sinh^2 \frac{\tau}{2}}.
\end{equation}
Limits where $\frac{k(n)}{(bn)^2} \to c$ for $c \neq \frac{1}{2\sinh^2 \frac{\tau}{2}}$ exist but are smaller.
\end{remark}

\begin{figure}
\includegraphics[width = .8\textwidth]{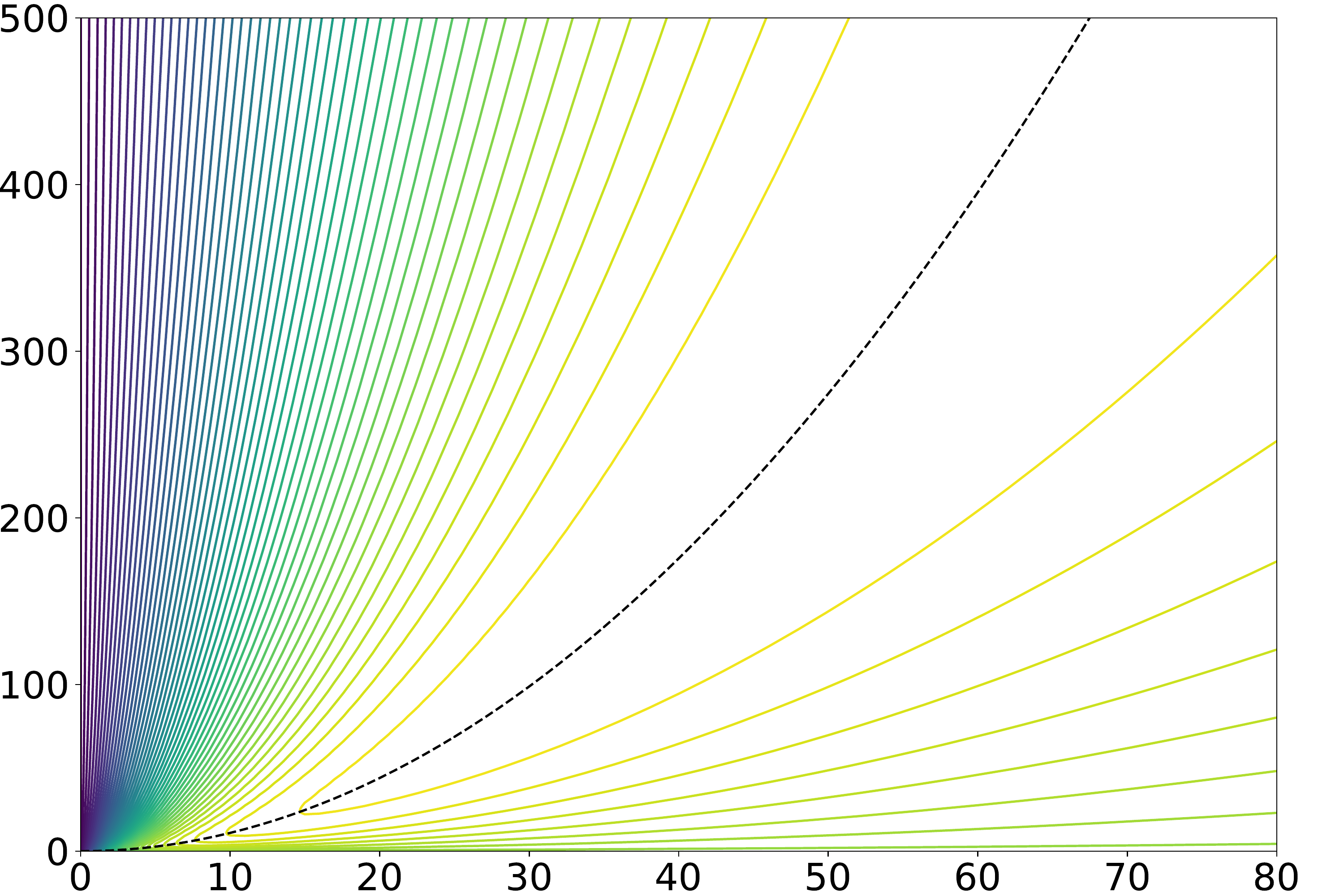}
\caption{Contours of $\frac{\log\|\Pi_{n/\sqrt{2}, k}\|}{n^2/4 + k}$ for $1 \leq n < 80$ (horizontal axis) and $1 \leq k < 500$ (vertical axis). Levels are $.02\tau m$ for $m = 0, 1, 2, \dots, 49$. Dotted line is $k = \frac{1}{2}c_0 n^2$ from \eqref{eq:sigma_concentration}.}
\label{fig:ratio}
\end{figure}

\section{An integral formula for $\|\Pi_{a, k}\|$}\label{s:integral}

Recall $\Pi_{a, k}$ from \eqref{eq:def_proj_hermite}, the spectral projection of the shifted harmonic oscillator $P_a$ from \eqref{eq:def_SHO} with $a \in \Bbb{R}\backslash \{0\}$, associated to the eigenvalue $k \in \Bbb{N}$. 

From \cite[Eq.~(2.18)]{Mityagin_Siegl_Viola_2016}, we know that with $h_k(x)$ are the Hermite functions \eqref{eq:def_hermite},
\begin{equation}\label{eq:norm_proj_hermite}
	\|\Pi_{a, k}\| = \|h_k(y - \ii a)\|^2_{L^2(\Bbb{R}_y)} = \|h_k(y)\ee^{-ay}\|^2_{L^2(\Bbb{R}_y)}.
\end{equation}
The latter equality comes from Fourier invariance of the Hermite functions and the fact that the standard identity
\begin{equation}
	(\fourier f(\cdot - b))(y) = \ee^{\ii b y} \fourier f(y)
\end{equation}
may be holomorphically extended to $b \in \Bbb{C}$ for extremely regular and rapidly decaying functions (like the Hermite functions).

One way of studying the Hermite functions is through the Bargmann transform $\barg$, whose definition we recall below in \eqref{eq:def_barg}. In particular, for 
\begin{equation}
	P_0 = \frac{1}{2}(-\frac{\dd^2}{\dd x^2} + x^2-1),
\end{equation}
\begin{equation}\label{eq:HO_barg}
	\barg P_0 \barg^* = x \frac{\dd}{\dd x},
\end{equation}
where the adjoint is with respect to the inner product \eqref{eq:HPhi0_ip}. The existence of a sequence of eigenvectors 
\begin{equation}
	\mathfrak{h}_k(x) = (\pi k !)^{-1/2}x^k, \quad k \in \Bbb{N},
\end{equation}
satisfying
\begin{equation}
	\barg P_0 \barg^*\mathfrak{h}_k(x) = k\mathfrak{h}_k(x)
\end{equation}
is then elementary. An exercise in integration using polar coordinates shows that these eigenvectors are orthonormal with respect to the inner product \eqref{eq:HPhi0_ip}.

We relate the eigenfunctions of the shifted harmonic oscillator to shifts of these Bargmann-side eigenfunctions, and translating back to the formula \eqref{eq:norm_proj_hermite} on $L^2(\Bbb{R})$, allows us to obtain the following formula for the norms of the spectral projections of the shifted harmonic oscillator.

\begin{theorem}\label{thm:norm_integral}
The $L^2$-operator norm of $\Pi_{a, k}$, the spectral projection of $P_a$ associated to $k \in \Bbb{N}$, is given by
\begin{equation}\label{eq:norm_integral}
	\|\Pi_{a, k}\| = \frac{\ee^{a^2}}{\pi k!}\iint \left((x_1+a\sqrt{2})^2 + x_2^2 \right)^k \ee^{-x_1^2 - x_2^2}\,\dd x_1 \,\dd x_2.
\end{equation}
\end{theorem}

\begin{remark}\label{rem:norm_fock}
With the Fock-space norm \eqref{eq:def_Fock} and $\gamma =   a\sqrt{2}$, we could equally well write
\begin{equation}\label{eq:norm_fock}
	\|\Pi_{\gamma/\sqrt{2}, k}\| = \frac{\ee^{\gamma^2/2}}{\pi k!} \|(x+\gamma)^k\|_\Fock^2.
\end{equation}
\end{remark}

\begin{proof}
For $x \in \Bbb{C}$ and $f \in L^2(\Bbb{R})$, the Bargmann transform is (up to certain choices of normalization)
\begin{equation}\label{eq:def_barg}
	\mathfrak{B}f(x) = \pi^{-3/4}\int_{\Bbb{R}} \ee^{\sqrt{2}xy - \frac{x^2}{2} - \frac{y^2}{2}}f(y)\,\dd y.
\end{equation}
We also define the phase-space shifts for $(u, v) \in \Bbb{C}^2$
\begin{equation}\label{eq:def_shift}
	\shift_{(u, v)}f(x) = \ee^{-\frac{\ii}{2}u v + \ii v x}f(x-u).
\end{equation}
(When $(u, v)$ is not real, the shift $\shift_{(u, v)}$ is not necessarily easy to define on $L^2(\Bbb{R})$ but if $f(x) = p(x)\ee^{-x^2/2}$ is a polynomial times a Gaussian, $\shift_{(u, v)}f(x)$ is well-defined and in $L^2(\Bbb{R})$ for any $(u, v) \in \Bbb{C}^2$.)

We use the following standard facts (see, e.g.,~\cite[Sec.~I.6, I.7]{Folland_1989}):
\begin{enumerate}[(a)]
\item If $(u, v) \in \Bbb{R}^2$, then $\shift_{(u, v)}$ is unitary on $L^2(\Bbb{R})$.
\item Shifts compose according to the following rule:
\begin{equation}\label{eq:shift_compose}
	\shift_{(u, v)} \shift_{(q, p)} = \ee^{\frac{\ii}{2}(vq - pu)}\shift_{(u+q, v+p)}
\end{equation}
\item The Bargmann transform is unitary from $L^2(\Bbb{R})$ onto the Fock space
\begin{equation}\label{eq:def_Fock}
	\begin{aligned}
	\Fock &= \{f\in \operatorname{Hol}(\Bbb{C}) \::\: \|f\|_{\Fock} < \infty\}, 
	\\
	\|f\|_{\mathfrak{F}}^2 &= \iint |f(x_1 + \ii x_2)|^2 \ee^{-x_1^2-x_2^2}\,\dd x_1\,\dd x_2.
	\end{aligned}
\end{equation}
The Fock space is equipped with the inner product 
\begin{equation}\label{eq:HPhi0_ip}
	\langle f, g\rangle_{\Fock} = \iint f(x_1 + \ii x_2)\overline{g(x_1 + \ii x_2)}\ee^{-x_1^2 - x_2^2}\,\dd x_1 \,\dd x_2,
\end{equation}
so for all $f, g \in L^2(\Bbb{R})$,
\begin{equation}\label{eq:Barg_unitary}
	\langle f, g\rangle_{L^2(\Bbb{R})} = \langle \barg f, \barg g\rangle_{\Fock}.
\end{equation}
\item The Bargmann transform transforms shifts according to the rule
\begin{equation}\label{eq:Barg_Egorov}
	\barg \shift_{(u, v)} = \shift_{\Bff{B}(u, v)}\barg, \quad \Bff{B} = \frac{1}{\sqrt{2}}\begin{pmatrix} 1 & -\ii \\ -\ii & 1\end{pmatrix}.
\end{equation}
\item\label{it:shift_unitary} A shift $\shift_{(u, v)}$ with $(u, v) \in \Bbb{C}^2$ is unitary on $\Fock$ if and only if $v = -\ii \bar{u}$.
\item When $\{h_k\}_{k\in\Bbb{N}}$ are the Hermite functions \eqref{eq:def_hermite}, for any $k \in \Bbb{N}$,
\begin{equation}\label{eq:barg_hermite}
	\barg h_k(x) = (\pi k!)^{-1/2} x^k.
\end{equation}
\end{enumerate}

Having recalled the essential elements of the theory of the Bargmann transform, we proceed with the proof. Using \eqref{eq:Barg_Egorov} and \eqref{eq:shift_compose},
\begin{equation}\label{eq:barg_space_shift}
	\shift_{\frac{1}{2}(\ii \gamma, -\gamma)}\barg \shift_{(\ii\gamma/\sqrt{2}, 0)}\barg^* = \shift_{\frac{1}{2}(\ii \gamma, -\gamma)}\shift_{\frac{1}{2}(\ii\gamma, \gamma)} = \ee^{\gamma^2/4}\shift_{(\ii\gamma, 0)},
\end{equation}
where we have composed on the left by $\shift_{\frac{1}{2}(\ii \gamma, -\gamma)}$ because it is unitary on $\Fock$ by \ref{it:shift_unitary}. Therefore, using also \eqref{eq:norm_proj_hermite}, \eqref{eq:Barg_unitary}, and \eqref{eq:barg_hermite},
\begin{equation}\label{eq:integral_1}
	\begin{aligned}
	\|\Pi_{\gamma/\sqrt{2}, k}\| &= \|\shift_{(\ii \gamma/\sqrt{2}, 0)}h_k\|^2 = \|\barg \shift_{(\ii \gamma/\sqrt{2}, 0)}h_k\|_\Fock^2 = (\pi k!)^{-1}\|\shift_{\frac{1}{2}(\ii \gamma, \gamma)}x^k\|_{\Fock}^2
	\\ &= (\pi k!)^{-1}\|\shift_{\frac{1}{2}(\ii\gamma, -\gamma)}\shift_{\frac{1}{2}(\ii\gamma, \gamma)}x^k\|_\Fock^2 = (\pi k!)^{-1}\ee^{\gamma^2/2}\|\shift_{(\ii\gamma, 0)}x^k\|_\Fock^2.
	\end{aligned}
\end{equation}

By the definition \eqref{eq:def_Fock} of the norm on the space $\Fock$,
\begin{equation}\label{eq:integral_2}
	\begin{aligned}
	\|\shift_{(\ii \gamma, 0)}x^k\|_\Fock^2 &= \iint (x_1^2 + (x_2 - \gamma)^2)^k \ee^{-x_1^2 - x_2^2}\,\dd x_1 \,\dd x_2
	\\ &= \iint ((x_1+\gamma)^2 + x_2^2)^k \ee^{-x_1^2 - x_2^2}\,\dd x_1 \,\dd x_2.
	\end{aligned}
\end{equation}

Combining \eqref{eq:integral_1} and \eqref{eq:integral_2} gives the result of the theorem.
\end{proof}

\section{Results and summary of applying Laplace's method}\label{s:laplace}

We use Laplace's method (see for instance \cite[Ch.\ 3]{Miller_book}) to approximate, in three stages, the result of Theorem \ref{thm:norm_integral} in polar coordinates: for $\gamma \in \Bbb{R}$,
\begin{equation}\label{eq:laplace_proj_integral_recall}
	\|\Pi_{\gamma/\sqrt{2}, k}\| = \frac{\ee^{\gamma^2/2}}{\pi k!}\int_{0}^\infty r\ee^{-r^2}\int_{-\pi}^\pi (r^2+\gamma^2+2\gamma r\cos\theta)^k\,\dd \theta\,\dd r.
\end{equation}

\subsection{Results}\label{ss:results} Throughout, we assume that $r > 0$ and that $\gamma > 0$ (since $\|\Pi_{\gamma, k}\| = \|\Pi_{-\gamma, k}\|$). Because our analysis involves tail estimates for Gaussians, we recall the error function
\begin{equation}\label{eq:def_erf}
	\erf(z) = \frac{2}{\sqrt{\pi}}\int_0^z \ee^{-t^2}\,\dd t.
\end{equation}
We also frequently use $k_1 = k + \frac{1}{2}$, which is convenient because we avoid dividing by zero (among other reasons).

The first result can be applied for any $k \in \Bbb{N}$ and $\gamma \in \Bbb{R}$ to find an asymptotic expression for $\|\Pi_{\gamma/\sqrt{2}, k}\|$, up to a factor of $\sqrt{k}$. We also observe a natural change of variables with $\gamma = \sqrt{2 k_1}\sinh u$ for $u \in \Bbb{R}$.

\begin{theorem}\label{thm:asymp_laplace_weak}
Let $k \in \Bbb{N}$, let $k_1 = k + \frac{1}{2}$ and let $u \geq 0$. Recall the spectral projection $\Pi_{a, k}$ from \eqref{eq:def_proj_hermite}. Then, with $C_0 = \sqrt{\pi}\erf(\frac{\pi}{2}) \approx 1.7258$,
\begin{equation}\label{eq:asymp_laplace_weak}
	\frac{C_0}{\ee\sqrt{8\pi k_1}}\ee^{-2u} \leq \ee^{-k_1(2u + \sinh 2u)}\|\Pi_{\sqrt{2k_1}\sinh u, k}\| \leq \sqrt{\frac{27}{\ee}}.
\end{equation}
\end{theorem}

We can sharpen this result for large $k$ with Laplace's method. In order to have a relative error which tends to zero, we consider $\gamma$ near $\sqrt{k}$, up to multiplication by a small positive power of $k$.

\begin{theorem}\label{thm:asymp_laplace_sharp}
Let $p_1, p_2 > 0$ satisfy $p_1 + p_2 < \frac{1}{2}$, and recall the spectral projection $\Pi_{a, k}$ from \eqref{eq:def_proj_hermite}. Then there exists $K > 0$ such that if $k \in \Bbb{N}$ with $k \geq K$ and if, writing $k_1 = k + \frac{1}{2}$,
\begin{equation}\label{eq:laplace_sharp_uhyp}
	\frac{3}{8}k_1^{-2p_2} \leq u \leq p_2 \log k_1,
\end{equation}
then
\begin{equation}
	\|\Pi_{\sqrt{2k_1}\sinh u, k}\| = (4\pi k_1 \sinh 2u)^{-1/2}\ee^{k_1(2u + \sinh 2u)}\left(1+\BigO(k_1^{-p_1})\right).
\end{equation}
\end{theorem}

\begin{remark}\label{rem:recover_MiSiVi}
If we fix $a > 0$ and define
\begin{equation}
	u = \operatorname{arsinh}\frac{a}{\sqrt{2k_1}},
\end{equation}
as $k \to \infty$,
\begin{equation}
	u = \frac{a}{\sqrt{2k_1}}(1+\BigO(k_1^{-1}))
\end{equation}
and
\begin{equation}
	\sinh 2u = 2\sinh u \cosh u = \frac{2a}{\sqrt{2k_1}}\sqrt{ 1 + \frac{a^2}{2k_1}} = \frac{a\sqrt{2}}{\sqrt{k}}(1+\BigO(k^{-1})).
\end{equation}
So long as $p_2 > \frac{1}{4}$, we are assured that $u \geq \frac{3}{8}k_1^{-2p_2}$ for $k$ sufficiently large. Therefore, for any $p_1 \in (0, \frac{1}{4})$ we can apply Theorem \ref{thm:asymp_laplace_sharp} choosing $p_2 \in (\frac{1}{4}, \frac{1}{2}-p_1)$. 

We obtain from Theorem \ref{thm:asymp_laplace_sharp} that
\begin{equation}
	\begin{aligned}
	\|\Pi_{a, k}\| 
	&= 
	\left(4\pi k_1\frac{a\sqrt{2}}{\sqrt{k}}(1+\BigO(k^{-1}))\right)^{-1/2}\ee^{4k_1\frac{a}{\sqrt{2k}}(1+\BigO(k^{-1}))}(1+\BigO(k^{-p_1}))
	\\ &= \frac{1}{2(2k)^{1/4}\sqrt{\pi a}}\ee^{2^{3/2}a\sqrt{k}}\ee^{\BigO(k^{-1/2})}(1+\BigO(k^{-p_1}))
	\\ &= \frac{1}{2(2k)^{1/4}\sqrt{\pi a}}\ee^{2^{3/2}a\sqrt{k}}(1+\BigO(k^{-p_1})).
	\end{aligned}
\end{equation}
Since $\|\Pi_{a, k}\| = \|\Pi_{|a|, k}\|$ for all $a \in \Bbb{R}$, we recover \eqref{eq:MiSiVi} up to a less sharp bound on the error (since $p_1 < \frac{1}{4}$). 

On the other hand, the estimates in Theorem \ref{thm:asymp_laplace_sharp} can be applied as $a$ varies; see Section \ref{ss:proj_to_laguerre}.
\end{remark}

\subsection{Strategy}\label{ss:laplace_strategy}

We proceed by summarizing the strategy used to prove Theorems \ref{thm:asymp_laplace_weak} and \ref{thm:asymp_laplace_sharp}, postponing the details to Sections \ref{s:laplace_theta}--\ref{s:laplace_proofs}.

We begin with the inner integral
\begin{equation}\label{eq:def_I1_theta}
	\begin{aligned}
	\int_{-\pi}^\pi (r^2+\gamma^2+2\gamma r\cos\theta)^k\,\dd \theta
	&= \int_{-\pi}^\pi \ee^{g(\theta)}\,\dd \theta
	\end{aligned}
\end{equation}
when
\begin{equation}\label{eq:def_gtheta}
	g(\theta) = k\log(r^2 + \gamma^2 + 2\gamma r\cos\theta).
\end{equation}
The function $g(\theta)$ is maximized when $\theta = 0$, and we approximate 
\begin{equation}\label{eq:def_A1_laplace}
	\int_{-\pi}^\pi \ee^{g(\theta)}\,\dd \theta \approx A_1(r, \gamma, k) := \sqrt{\frac{\pi}{kr\gamma}}(r+\gamma)^{2k_1},
\end{equation}
where we use the shorthand $k_1 = k + \frac{1}{2}$. (Here, $\approx$ is simply a heuristic: precise statements are in Lemmas \ref{lem:laplace_theta_weak} and \ref{lem:laplace_theta} below.)

Next, we proceed to analyze the integral in $r$. When the approximation \eqref{eq:def_A1_laplace} holds (under hypotheses specified in Propositions \ref{prop:laplace_app} and \ref{prop:laplace_r}),
\begin{equation}\label{eq:def_J2_laplace}
	\begin{aligned}
	\JJ &= \int_{0}^\infty r\ee^{-r^2}\int_{-\pi}^\pi (r^2+\gamma^2+2\gamma r\cos\theta)^k\,\dd \theta\,\dd r
	\\ &= \int_0^\infty r\ee^{-r^2}\int_{-\pi}^\pi \ee^{g(\theta)}\,\dd \theta\,\dd r
	\\ &\approx \int_0^\infty \sqrt{\frac{\pi r}{k \gamma}}(r+\gamma)^{2k_1}\ee^{-r^2}\,\dd r
	\\ &\approx \int_0^\infty \sqrt{\frac{\pi r}{k \gamma}}\ee^{G(r)}\,\dd r,
	\end{aligned}
\end{equation}
for
\begin{equation}\label{eq:def_Gr}
	G(r) = 2k_1\log(r+\gamma) - r^2.
\end{equation}
(Again, the precise meaning of $\approx$ is to be found in the statements of Propositions \ref{prop:laplace_app} and \ref{prop:laplace_r}.)

Using Laplace's method near the critical point of $G$
\begin{equation}\label{eq:def_r1}
	r_1 = -\frac{\gamma}{2} + \sqrt{k_1 + \frac{\gamma^2}{4}},
\end{equation}
we approximate the integral in $r$ by
\begin{equation}\label{eq:def_A2_laplace}
	\int_0^\infty \sqrt{\frac{\pi r}{k \gamma}}\ee^{G(r)}\,\dd r \approx A_2(\gamma, k) := \pi\sqrt{\frac{r_1}{\gamma k_1}}\left(1 + \frac{k_1}{(r_1+\gamma)^2}\right)^{-1/2}(r_1 + \gamma)^{2k_1}\ee^{-r_1^2}.
\end{equation}
This expression simplifies significantly when we introduce the notation
\begin{equation}\label{eq:def_u_arsinh}
	u = \arsinh\frac{\gamma}{2\sqrt{k_1}},
\end{equation}
which leads to the identities
\begin{equation}\label{eq:arsinh_identity1}
	\sqrt{k_1 + \frac{\gamma^2}{4}} = \sqrt{k_1}\sqrt{1 + \left(\frac{\gamma}{2\sqrt{k_1}}\right)^2} = \sqrt{k_1}\cosh u
\end{equation}
and
\begin{equation}\label{eq:arsinh_identities2}
	r_1 = \sqrt{k_1}\ee^{-u}, \quad r_1+\gamma = \sqrt{k_1}\ee^{u}, \quad \frac{r_1}{\gamma} = \frac{1}{\ee^{2u}-1}, \quad \frac{\gamma^2}{2} - r_1^2 + k_1 = k_1\sinh 2u.
\end{equation}
Using this notation, we can rewrite
\begin{equation}
	A_2(\gamma, k) = \pi(2k_1\sinh 2u)^{-1/2} (k_1\ee^{\ee^{-2u}})^{k_1}.
\end{equation}

Having established this approximation for $\JJ$, Theorems \ref{thm:asymp_laplace_weak} and \ref{thm:asymp_laplace_sharp} can be proven from \eqref{eq:norm_fock} and Stirling's approximation.

\subsection{Analyzing quotients}\label{ss:laplace_h}

To relate $\sigma(b)$ in Definition \ref{def:sigma} to Theorems \ref{thm:asymp_laplace_weak} and \ref{thm:asymp_laplace_sharp}, note that when $u = u(n, k_1, b)$ is defined by $bn = \sqrt{2k_1}\sinh u$,
\begin{equation}
	\sigma(b) = \limsup_{n \to \infty}\left(\sup_{k\in\Bbb{N}}\frac{\log\|\Pi_{\sqrt{2k_1}\sinh u, k}\|}{k_1 \sinh^2 u + k_1}\right) = \limsup_{n \to \infty}\left(\sup_{k\in\Bbb{N}}\frac{\log\|\Pi_{\sqrt{2k_1}\sinh u, k}\|}{k_1 \cosh^2 u}\right).
\end{equation}

Letting $u\geq 0$ (since the result is the same if we exchange $u$ and $-u$), from Theorem \ref{thm:asymp_laplace_weak} and \eqref{eq:asymp_laplace_weak} there exist $C_1, C_2 > 0$ such that
\begin{equation}
	- C_1 - \frac{1}{2}\log k_1 - 2u - 1 
	\leq 
	\log\|\Pi_{\sqrt{2k_1}\sinh u, k}\| - k_1(2u + \sinh 2u) 
	\leq
	C_2.
\end{equation}
If we define
\begin{equation}\label{eq:hu}
	h(u) = \frac{u + \sinh u \cosh u}{\cosh^2 u},
\end{equation}
then for all $u \geq 0$ and all $k \in \Bbb{N}$,
\begin{equation}\label{eq:sigma_hu}
	\begin{aligned}
	- \frac{1}{k_1\cosh^2 u}\left(\frac{1}{2}\log k_1 + 2u + C_1 + 1\right) &\leq \frac{\log\|\Pi_{\sqrt{2k_1}\sinh u, k}\|}{k_1 \cosh^2 u} - 2h(u) 
	\\ &\leq \frac{1}{k_1\cosh^2 u}C_2.
	\end{aligned}
\end{equation}

\begin{figure}
\includegraphics[width = .8\textwidth]{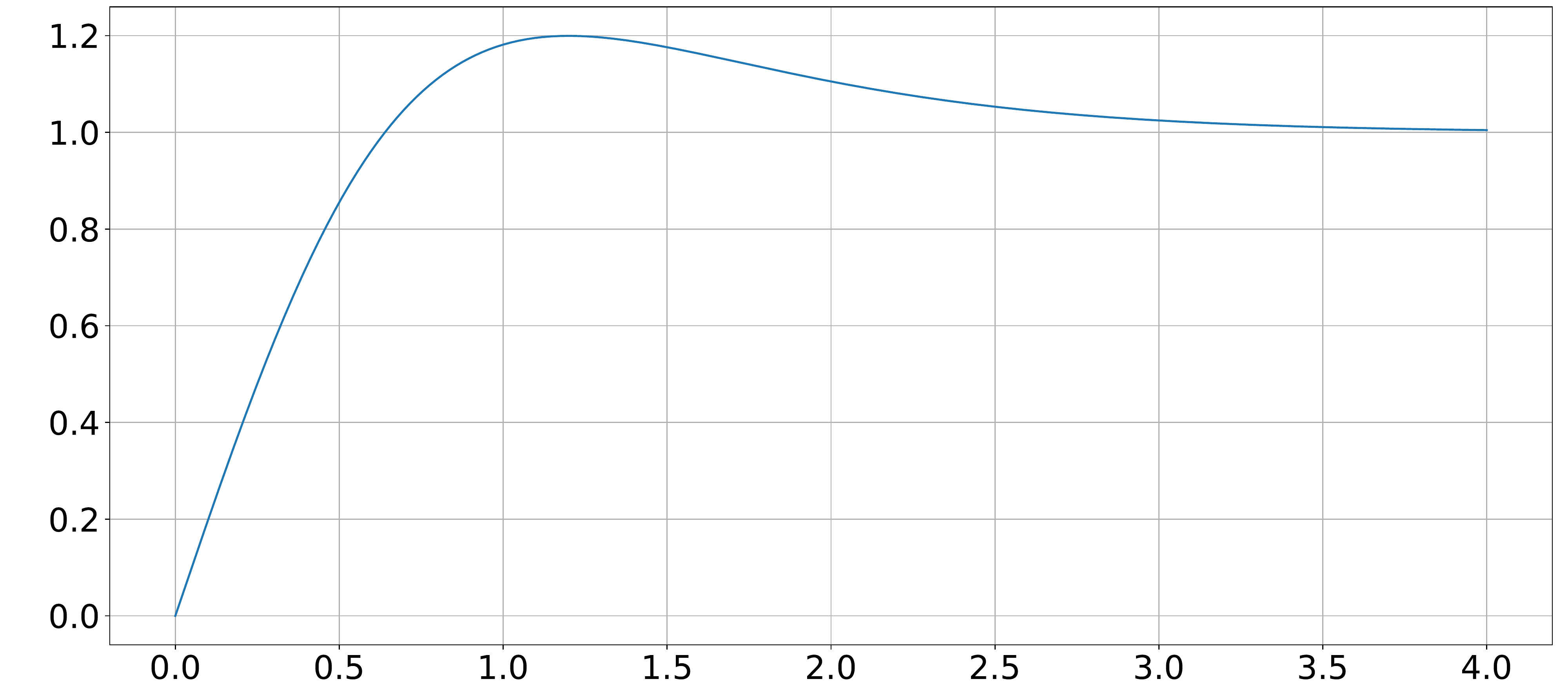}
\caption{Plot of function $h(u)$ from \eqref{eq:hu}.}
\end{figure}

\begin{remark}
\begin{equation}
	\lim_{u \to +\infty} h(u) = 1,
\end{equation}
we can deduce from \eqref{eq:sigma_hu} that when $k$ is fixed,
\begin{equation}
	\lim_{u \to \infty}\frac{\log\|\Pi_{\sqrt{2k_1}\sinh u, k}\|}{k_1 \cosh^2 u} = 2.
\end{equation}
We confirm this fact by applying \eqref{eq:MiSiVi_laguerre}:
\begin{equation}
	\lim_{u \to \infty} \frac{\log\|\Pi_{\sqrt{2k_1}\sinh u, k}\|}{k_1 \cosh^2 u} 
	= 
	\lim_{u \to \infty}\frac{2k_1\sinh^2 u + \log L_k^{(0)}(-4k_1(\sinh u)^2)}{k_1 \cosh^2 u} 
	= 
	2.
\end{equation}
\end{remark}

We compute the derivative
\begin{equation}
	\begin{aligned}
	h'(u) &= \frac{1 + \cosh^2 u + \sinh^2 u}{\cosh^2 u} - \frac{2(u + \sinh u \cosh u)\sinh u}{\cosh^3 u}
	\\ &= \frac{1}{\cosh^3 u}(2\cosh^3 u - 2\sinh^2 u \cosh u - 2u \sinh u)
	\\ &= \frac{2}{\cosh^3 u}(\cosh u - u\sinh u).
	\end{aligned}
\end{equation}
For $u \geq 0$, this function vanishes if and only if $u = u_0 \approx 1.1997$ is the unique positive fixed point of the hyperbolic cotangent:
\begin{equation}\label{eq:coth_fixed}
	u_0 = \coth u_0.
\end{equation}
We recall from Definition \ref{def:tau} that $u_0 = \frac{\tau}{2}$.

Using that $u_0$ satisfies \eqref{eq:coth_fixed}, we compute
\begin{equation}
	h(u_0) = \frac{\coth u_0 + \sinh u_0 \cosh u_0}{\cosh^2 u_0} = \frac{1+\sinh^2 u_0}{\sinh u_0\cosh u_0} = \frac{\cosh u_0}{\sinh u_0} = u_0.
\end{equation}
We summarize the relevant properties of $h$ in the following proposition.

\begin{proposition}\label{prop:hu0}
Let
\begin{equation}
	h(u) = \frac{u + \sinh u \cosh u}{\cosh^2 u}
\end{equation}
and let $\tau$ be as in Definition \ref{def:tau}. Then $h$ is an odd function, $h(u)$ is increasing on $(0, \frac{\tau}{2})$ from $h(0) = 0$ to $h(\frac{\tau}{2}) = \frac{\tau}{2}$, and $h(u)$ is decreasing on $(\frac{\tau}{2}, \infty)$ with $\lim_{u \to \infty} h(u) = 1$.
\end{proposition}

\subsection{Proof of Theorem \ref{thm:sigma_tau}}\label{ss:proof_sigma_tau}

Recall that if $bn = \sqrt{2k_1}\sinh u$, then $k_1 \cosh^2 u = \frac{1}{2}(bn)^2 + k_1$. The upper bound in \eqref{eq:sigma_hu} gives
\begin{equation}
	\sup_{k \in \Bbb{N}} \frac{\log\|\Pi_{bn, k}\|}{\frac{1}{2}(bn)^2 + k_1} \leq \tau + \frac{C_2}{\frac{1}{2}(bn)^2 + k_1},
\end{equation}
so
\begin{equation}\label{eq:sigma_tau_ub}
	\limsup_{n\to \infty} \left(\sup_{k\in\Bbb{N}}\frac{\log\|\Pi_{bn, k}\|}{\frac{1}{2}(bn)^2 + k_1}\right) \leq \tau.
\end{equation}

To show that the $\limsup$ is at least $\tau$, it suffices to take a sequence $k(n) \in \Bbb{N}$ such that 
\begin{equation}\label{eq:sigma_concentration}
	\lim_{n\to\infty} \frac{k(n)+\frac{1}{2}}{(bn)^2} = c_0, \quad c_0 = \frac{1}{2\sinh^2 \frac{\tau}{2}} \approx 0.21961.
\end{equation}
Letting $k(n)$ be the greatest integer less than $c_0(bn)^2 - \frac{1}{2}$ for $n$ large certainly suffices; note also that this assumption means that $\lim_{n\to\infty}k(n) = \infty$. By continuity of $\arsinh$,
\begin{equation}
	u(n) := \arsinh \frac{bn}{\sqrt{2(k(n)+\frac{1}{2})}} \to \arsinh \sqrt{2c_0} = \frac{\tau}{2}, \quad n \to \infty.
\end{equation}
In this case by \eqref{eq:sigma_hu},
\begin{equation}
	\begin{aligned}
	\limsup_{\min\{|n|,k\}\to \infty} \frac{\log\|\Pi_{bn, k}\|}{\frac{1}{2}(bn)^2 + k_1}
	& \geq
	\lim_{n \to \infty} \frac{\log\|\Pi_{bn, k(n)}\|}{\frac{1}{2}(bn)^2 + k(n) + \frac{1}{2}} 
	\\ & \geq 
	\lim_{n \to \infty}\left(2h(u(n)) - \frac{\frac{1}{2}\log (k(n)+\frac{1}{2}) + 2u(n) + C_1}{\frac{1}{2}(bn)^2 + k(n)+\frac{1}{2}}\right)
	\\ &= 
	2h(\frac{\tau}{2}) = \tau.
	\end{aligned}
\end{equation}
Taken with \eqref{eq:sigma_tau_ub}, this completes the proof of the theorem.

\begin{remark}
Suppose, for instance, $k = k(n) \in \Bbb{N}$ satisfies for some $C > 0$ the estimate
\begin{equation}
	\left|k(n) - c_0(bn)^2\right| \geq \frac{1}{C}n^2,
\end{equation}
One could then use \eqref{eq:sigma_hu} to show that
\begin{equation}
	\limsup_{n\to\infty} \frac{\log\|\Pi_{bn, k(n)}\|}{\frac{1}{2}(bn)^2 + k(n)} < \tau.
\end{equation}
\end{remark}

\section{Laguerre polynomials and combinatorics}\label{s:laguerre}

As shown in \cite[Eq.~(2.20)]{Mityagin_Siegl_Viola_2016}, the spectral projections for the shifted harmonic oscillator can be written in terms of the Laguerre polynomials. In Section \ref{ss:laguerre_comb} we use a combinatorial approach elementary estimates on the Laguerre polynomials; these estimates are far from optimal and it would be interesting to see whether they can be improved. In Section \ref{ss:proj_to_laguerre} we apply Theorems \ref{thm:asymp_laplace_weak} and \ref{thm:asymp_laplace_sharp} to obtain asymptotics for the Laguerre polynomials $L^{(0)}_k(-\gamma^2)$ as $k \to \infty$, for a broad range of $\gamma$ around $\sqrt{k}$.

\subsection{Asymptotics as $k \to \infty$}\label{ss:laguerre_comb}

\begin{proposition}\label{prop:laguerre_ests}
For the Laguerre polynomials $L^{(0)}_k$ defined in \eqref{eq:def_laguerre} and for every $\gamma \in \Bbb{R}$,
\begin{equation}\label{eq:laguerre_ub}
	L^{(0)}_k(-\gamma^2) \leq \cosh(2\gamma\sqrt{k}).
\end{equation}
Furthermore, for any $s \in (0, 1)$, if $k \geq \frac{1}{s}$ and
\begin{equation}\label{eq:laguerre_lb_hyp}
	0 \leq \gamma \leq \frac{1}{\ee}\frac{s}{\sqrt{1-s}}\sqrt{k},
\end{equation}
then
\begin{equation}\label{eq:laguerre_lb}
	L^{(0)}_k(-\gamma^2) \geq \frac{2}{5\sqrt{sk}}\cosh(2\gamma\sqrt{(1-s)k}).
\end{equation}
\end{proposition}

\begin{proof}
If we write
\begin{equation}
	P_m = \frac{(2m)!}{(2^m m!)^2}, \quad m \in \Bbb{N},
\end{equation}
then 
\begin{equation}\label{eq:laguerre_Pm}
	L^{(0)}_k(-\gamma^2) = \sum_{m=0}^k \frac{k!}{(m!)^2 (k-m)!}\frac{(2^m m!)^2}{(2m)!}P_m\gamma^{2m} = \sum_{m=0}^k \frac{(2\gamma)^{2m}}{(2m)!}\frac{k!}{(k-m)!}P_m.
\end{equation}
Using the notation $n!!$ for the product of integers $1 \leq j \leq n$ having the same parity as $n$,
\begin{equation}
	P_m = \frac{(2m-1)!!}{(2m)!!} = \prod_{j=1}^m \left(1-\frac{1}{2j}\right),
\end{equation}
from which we see that $P_m \leq 1$. Since $\frac{k!}{(k-m)!} \leq k^m$, from \eqref{eq:laguerre_Pm} we can deduce
\begin{equation}
	L^{(0)}_k(-\gamma^2) \leq \sum_{m=0}^k \frac{1}{(2m)!}(2\gamma\sqrt{k})^{2m} \leq \cosh(2\gamma\sqrt{k}).
\end{equation}

In the other direction, for $s \in (0, 1)$ fixed, when $m \leq sk$
\begin{equation}
		\frac{k!}{(k-m)!} = \prod_{j=0}^{m-1} (k-j) \geq \left((1-s)k\right)^m.
\end{equation}
We can estimate $P_m$ by noting that, since $\frac{2t}{2t-1} \leq \sqrt{\frac{t}{t-1}}$ for all $t > 1$,
\begin{equation}
	P_m = \prod_{j = 1}^m \frac{2j-1}{2j} \geq \frac{1}{2} \prod_{j=2}^m \sqrt{\frac{j-1}{j}} = \frac{1}{2\sqrt{m}}.
\end{equation}
In particular, if $sk \geq 1$ and $0 \leq m \leq sk$, then $P_m \geq \frac{1}{2\sqrt{sk}}$. Therefore from \eqref{eq:laguerre_Pm} we obtain whenever $sk \geq 1$ that
\begin{equation}\label{eq:laguerre_lb_1}
	L^{(0)}_k(-\gamma^2) \geq \frac{1}{2\sqrt{sk}}\sum_{m=0}^{sk}\frac{1}{(2m)!}
\end{equation}
\begin{equation}
	L^{(0)}_k(-\gamma^2) \geq \sum_{m=0}^{sk}P_m\frac{(2\gamma\sqrt{(1-s)k})^{2m}}{(2m)!}.
\end{equation}
Whenever $sk \geq 1$,
\begin{equation}
	L^{(0)}_k(-\gamma^2) \geq \frac{1}{2\sqrt{sk}}\sum_{m=0}^{sk} \frac{(2\gamma\sqrt{(1-s)k})^{2m}}{(2m)!}.
\end{equation}

Since $(2(N+1+n))! \geq (2(N+1))! (2n)!$, for all $x \geq 0$ and $N \in \Bbb{N}$,
\begin{equation}
	\cosh x - \sum_{m=0}^N \frac{x^{2m}}{(2m)!} = \sum_{m=0}^\infty \frac{x^{2(N+1+m)}}{(2(N+1+m))!} \leq \frac{x^{2(N+1)}}{(2(N+1))!}\cosh x.
\end{equation}
By Stirling's approximation \eqref{eq:Stirling_csts},
\begin{equation}
	\frac{x^{2(N+1)}}{(2(N+1))!} \leq \frac{1}{\sqrt{4\pi(N+1)}}\left(\frac{\ee x}{2(N+1)}\right)^{2(N+1)}.
\end{equation}
Therefore if
\begin{equation}\label{eq:cosh_taylor_xN}
	x \leq \frac{2}{\ee}(N+1),
\end{equation}
then
\begin{equation}\label{eq:cosh_taylor_relerr}
	\frac{1}{\cosh x} \left(\cosh x - \sum_{m = 0}^{N} \frac{x^{2m}}{(2m)!}\right) \leq \frac{1}{\sqrt{4\pi(N+1)}}.
\end{equation}

We apply this estimate to \eqref{eq:laguerre_lb_1} with $x = 2\gamma \sqrt{(1-s)k}$ and $N = sk$, while we still assume that $sk \geq 1$. The assumption \eqref{eq:laguerre_lb_hyp} was made in order to assure that \eqref{eq:cosh_taylor_xN} is satisfied. Having already assumed that $k \geq \frac{1}{s}$, 
\begin{equation}
	\frac{1}{\sqrt{4\pi(N+1)}} \leq \frac{1}{\sqrt{8\pi}} \leq \frac{1}{5}.
\end{equation}
Therefore, combining \eqref{eq:laguerre_lb_1} and \eqref{eq:cosh_taylor_relerr},
\begin{equation}
	\begin{aligned}
	L^{(0)}_k(-\gamma^2) &\geq \frac{1}{2\sqrt{N}}\cosh x\left(1 + \frac{1}{\cosh x}\left(\sum_{m=0}^{N} \frac{x^{2m}}{(2m)!} - \cosh x\right)\right)
	\\ &\geq \frac{2}{5\sqrt{N}}\cosh x = \frac{2}{5\sqrt{sk}}\cosh(2\gamma\sqrt{(1-s)k}).
	\end{aligned}
\end{equation}
This completes the proof of the proposition.
\end{proof}

From \eqref{eq:laguerre_ub} and the fact that $\cosh x \leq \ee^{|x|}$, for $t > 0$
\begin{equation}
\begin{aligned}
	\ee^{-t(\frac{1}{2}(bn)^2 + k)} \|\Pi_{bn, k}\| 
	&= \ee^{-t(\frac{1}{2}(bn)^2 + k)}\ee^{(bn)^2}L_k(-2(bn)^2) 
	\\ &\leq \exp\left(-t(\frac{1}{2}(bn)^2 + k)+(bn)^2 + 2^{3/2}|bn|\sqrt{k}\right).
\end{aligned}
\end{equation}
The exponent is a quadratic form in $bn, \sqrt{k}$ which is negative definite when $t > 1+\sqrt{5}$. Therefore, by comparison with the integral of an integrable Gaussian, if $t > 1 + \sqrt{5}$,
\begin{equation}
	\sum_{n \in \Bbb{Z}} \sum_{k \in \Bbb{N}}\ee^{-t(\frac{1}{2}(bn)^2 + k)}\|\Pi_{bn, k}\| < \infty.
\end{equation}
As a result, we have the following upper bound for $\sigma(b)$, which is rather far from the optimal lower bound $\tau$ from Definition \ref{def:tau}.

\begin{corollary}
For $\sigma$ defined in Definition \ref{def:sigma},
\begin{equation}
	\sigma(b) \leq 1+\sqrt{5}.
\end{equation}
\end{corollary}

\begin{remark}
If we look for a corresponding lower bound, we apply \eqref{eq:laguerre_lb} with $\gamma = \sqrt{2}|bn|$. Let $s \in (0, 1)$ be fixed, and suppose that $k \geq \frac{1}{s}$, and
\begin{equation}\label{eq:rbnk}
	\sqrt{2}|bn| \leq \frac{s}{\ee\sqrt{1-s}}\sqrt{k}.
\end{equation}
Using \eqref{eq:MiSiVi_laguerre}, \eqref{eq:laguerre_lb}, and $\cosh x \geq \frac{1}{2}\ee^{|x|}$, we obtain
\begin{equation}\label{eq:laguerre_lb_app}
	\ee^{-t(\frac{1}{2}(bn)^2 + k)}\|\Pi_{bn, k}\| \geq \frac{1}{5\sqrt{sk}}\ee^{-t(\frac{1}{2}(bn)^2 + k) + (bn)^2 + 2^{3/2}|bn|\sqrt{(1-s)k}}.
\end{equation}
If we set $r = \sqrt{k}/|bn|$, the exponent is the quadratic function
\begin{equation}
	\varphi(r) = (bn)^2 \left((1- \frac{t}{2}) + 2^{3/2}\sqrt{(1-s)}r - tr^2\right).
\end{equation}
	
If, for a given $b > 0$ and $t > 0$, there exists	
\begin{equation}\label{eq:laplace_lbapp_rnecc}	
	r > \frac{\ee \sqrt{2(1-s)}}{s}	
\end{equation}	
such that $\varphi(r) > 0$, then we could take a sequence $n(k)$ tending to infinity such that 	
\begin{equation}	
	\lim_{k \to \infty}\ee^{-t(\frac{1}{2}(bn(k))^2 + k)}\|\Pi_{bn(k), k}\| \to \infty,	
\end{equation}	
proving that $\sigma(b) \geq t$. However, the zeros of $\varphi(r)$ are	
	
\begin{equation}	
	r = -\sqrt{2(1-s)} \pm \sqrt{2(1-s) + t(1-\frac{t}{2})},	
\end{equation}	
and therefore when $t > 2$ there is no $r > 0$ such that $\varphi(r) > 0$. The estimate \eqref{eq:laguerre_lb} apparently cannot improve on Proposition \ref{prop:sigma_nk_fixed}, and better estimates on asymptotics for $L^{(0)}_k$ would be needed. We have these asymptotics, albeit indirectly, in Corollaries \ref{cor:asymp_laguerre_weak} and \ref{cor:asymp_laguerre_sharp} below.	
\end{remark}

\subsection{Applications of spectral projection asymptotics to Laguerre polynomial asymptotics}\label{ss:proj_to_laguerre}

In this work we can reverse the idea of the proof of \cite[Thm.~2.6]{Mityagin_Siegl_Viola_2016}: instead of using asymptotics for the Laguerre polynomials to prove an asymptotic formula for the spectral projection norms of the shifted harmonic oscillator, we can use asymptotics for the spectral projection norms to deduce asymptotics for the Laguerre polynomials. Recall that, for $\gamma \in \Bbb{R}$,
\begin{equation}
	L^{(0)}_k(-\gamma^2) = \ee^{-\gamma^2/2}\|\Pi_{\gamma/\sqrt{2}, k}\|
\end{equation}
where $L^{(0)}_k$ is a Laguerre polynomial \eqref{eq:def_laguerre} and $\Pi_{\gamma/\sqrt{2}, k}$ is the spectral projection \eqref{eq:norm_proj_hermite}.

Theorems \ref{thm:asymp_laplace_weak} and \ref{thm:asymp_laplace_sharp} are in terms of 
\begin{equation}
	u = \arsinh \frac{\gamma}{2\sqrt{k_1}}, \quad k_1 = k + \frac{1}{2}.
\end{equation}
We note in particular that
\begin{equation}
	\sinh 2u = 2\cosh u \sinh u = 2\sqrt{1 + \frac{\gamma^2}{4k_1}}\frac{\gamma}{2\sqrt{k_1}} = \frac{1}{k_1}\gamma \sqrt{k_1 + \frac{\gamma^2}{4}}.
\end{equation}

\begin{corollary}\label{cor:asymp_laguerre_weak}
If $k \in \Bbb{N}$ and $\gamma > 0$, with $C_0 = \sqrt{\pi}\erf(\frac{\pi}{2})$,
\begin{equation}
	\frac{C_0}{\ee\sqrt{8\pi k_1}}\ee^{-2\arsinh \frac{\gamma}{2\sqrt{k_1}}} \leq \ee^{\frac{\gamma^2}{2} - \gamma \sqrt{k_1 + \frac{\gamma^2}{4}} - k_1 \arsinh \frac{\gamma}{2\sqrt{k_1}}}L^{(0)}_k(-\gamma^2) \leq \sqrt{\frac{27}{\ee}}.
\end{equation}
\end{corollary}

\begin{corollary}\label{cor:asymp_laguerre_sharp}
Let $p_1, p_2 > 0$ satisfy $p_1 + p_2 < \frac{1}{2}$. There exists some $K > 0$ such that for all $k \in \Bbb{N}$ satisying $k \geq K$ and for all $\gamma > 0$ satisfying
\begin{equation}
	\frac{3}{8}k_1^{-2p_2} \leq \arsinh \frac{\gamma}{2\sqrt{k_1}} \leq p_2 \log k_1,
\end{equation}
\begin{equation}
	L^{(0)}_k(-\gamma^2) = \left(4\pi \gamma\sqrt{k_1 + \frac{\gamma^2}{4}}\right)^{-1/2}\ee^{-\frac{\gamma^2}{2} + \gamma \sqrt{k_1 + \frac{\gamma^2}{4}} + k_1 \arsinh \frac{\gamma}{2\sqrt{k_1}}}(1 + \BigO(k^{-p_1})).
\end{equation}
\end{corollary}

\begin{remark}
Because $\lim_{x \to 0^+}\frac{\sinh x}{x} = 1$ and $\lim_{x \to +\infty} \frac{\sinh x}{\ee^{x}/2} = 1$, for any $0 < \eps < \frac{1}{2}$ we can apply Corollary \ref{cor:asymp_laguerre_sharp} to
\begin{equation}
	\frac{\gamma}{\sqrt{2}} \in \left[2k_1^{-\frac{1}{2}+2\eps}, \frac{1}{2} k_1^{1-\eps}\right]
\end{equation}
when $k \geq K$ for $K>0$ sufficiently large, simply by taking $p_2 \in (\frac{1}{2}-\eps, \frac{1}{2})$.
\end{remark}

\section{Applying Laplace's method to the integral in $\theta$}\label{s:laplace_theta}

We begin the proof of Theorems \ref{thm:asymp_laplace_weak} and \ref{thm:asymp_laplace_sharp} by analyzing	
\begin{equation}	
	\int_{-\pi}^\pi \ee^{g(\theta)}\,\dd \theta, \quad g(\theta) = k\log(r^2 + \gamma^2 + 2\gamma r \cos\theta)	
\end{equation}	
as in \eqref{eq:def_I1_theta} and \eqref{eq:def_gtheta}. We suppose throughout that $r, \gamma > 0$ and $k \in \Bbb{N}$.

We compute the first two derivatives of $g(\theta)$:
\begin{equation}
	g'(\theta) = -\frac{2\gamma kr\sin\theta}{r^2 + 2\gamma r\cos\theta + \gamma^2}
\end{equation}
and
\begin{equation}
	\begin{aligned}
	g''(\theta) &= -\frac{2\gamma kr \cos\theta}{r^2 + 2\gamma r\cos\theta + \gamma^2} - k\left(\frac{2\gamma r\sin\theta}{r^2 + 2\gamma r\cos\theta + \gamma^2}\right)^2
	\\ &= -\frac{2\gamma kr(\cos\theta(r^2 + \gamma^2) + 2\gamma r)}{(r^2 + 2\gamma r\cos\theta + \gamma^2)^2}.
	\end{aligned}
\end{equation}
Note that $g(\theta)$ is even in $\theta$ and decreasing for $\theta \in [0, \pi]$, so
\begin{equation}
	g(\theta) \leq g(\Theta), \quad \forall \Theta \in [0, \pi], \quad \forall \theta \in [-\Theta, \Theta].
\end{equation}

For the second derivative, we consider the function
\begin{equation}\label{eq:def_Mrc}
	M(r, \gamma) = \frac{\gamma r(\cos \theta (r^2 + \gamma^2) + 2\gamma r)}{(r^2 + 2\gamma r\cos \theta + \gamma^2)^2} = \frac{\gamma r(\cos\theta (r+\gamma)^2 + 2(1-\cos\theta)\gamma r)}{(\cos\theta(r+\gamma)^2 + (1-\cos\theta)(r^2+\gamma^2))^2},
\end{equation}
symmetric and homogeneous of degree zero for $(r, \gamma) \in (0, \infty)^2$. Since
\begin{equation}\label{eq:Mrc_ddg}
	g''(\theta) = -2kM(r, \gamma),
\end{equation}
estimating $M(r, \gamma)$ allows us to estimate $g''(\theta)$, showing that $g(\theta)$ concentrates around $\theta = 0$ as $k \to \infty$.

We begin with estimates which apply without further hypotheses on $r, \gamma > 0$ and $k \in \Bbb{N}$.

\begin{lemma}\label{lem:laplace_theta_weak}
For $r, \gamma > 0$ and $k \in \Bbb{N}$, with $g(\theta)$ defined in \eqref{eq:def_gtheta},
\begin{equation}\label{eq:laplace_ub_weak}
	\int_{-\pi}^\pi \ee^{g(\theta)}\,\dd \theta \leq 2\pi(r+\gamma)^{2k},
\end{equation}
and if in addition $k \geq 1$,
\begin{equation}\label{eq:laplace_lb_weak}
	\int_{-\pi}^\pi \ee^{g(\theta)}\,\dd \theta \geq \frac{C_0}{\sqrt{k}}(r+\gamma)^{2k},
\end{equation}
where $C_0 = \sqrt{\pi}\erf(\frac{\pi}{2}) \approx 1.7258$.
\end{lemma}

\begin{proof}
The first inequality comes simply from $\ee^{g(\theta)}\leq \ee^{g(0)}$ for all $\theta \in [-\pi, \pi]$.

For the second inequality, when $|\theta| \leq \pi/2$,
\begin{equation}
	M(r, \gamma) \leq \frac{\gamma r(r^2+\gamma^2+2r\gamma)^2}{(r^2 + \gamma^2)^2} \leq \frac{\frac{1}{2}(r^2 + \gamma^2)2(r^2 + \gamma^2)}{(r^2 + \gamma^2)^2} = 1.
\end{equation}
Therefore for all $\theta \in [-\frac{\pi}{2}, \frac{\pi}{2}]$, $g''(\theta) \geq -2k$, so
\begin{equation}
	g(\theta) \geq g(0) - k\theta^2, \quad \forall \theta \in \left[-\frac{\pi}{2}, \frac{\pi}{2}\right].
\end{equation}
We conclude that
\begin{equation}
	\int_{-\pi}^\pi \ee^{g(\theta)}\,\dd \theta \geq \int_{-\pi/2}^{\pi/2} \ee^{g(0) - k\theta^2}\,\dd \theta \geq \frac{\ee^{g(r,0)}}{\sqrt{k}}\int_{-\pi\sqrt{k}/2}^{\pi\sqrt{k}/2} \ee^{-\theta^2}\,\dd \theta.
\end{equation}
The integral is bounded from below by the integral when $k = 1$, which can be expressed using the error function \eqref{eq:def_erf}. Since $\ee^{g(0)} = (r+\gamma)^{2k}$, we have completed the proof of the lemma.
\end{proof}

The estimates in Lemma \ref{lem:laplace_theta_weak} apply without supplementary hypotheses on $r, \gamma,$ and $k$, but leave a gap between upper and lower bounds. We can improve these estimates, especially for $k$ large, by consider $\theta \in [-\Theta, \Theta]$ for a varying $\Theta \in (0, \frac{\pi}{2})$.

\begin{lemma}\label{lem:laplace_theta}
For $r, \gamma > 0$, $\Theta \in (0, \frac{\pi}{2})$, and $k \in \Bbb{N}\backslash\{0\}$, let
\begin{equation}\label{eq:def_rhorc}
	\rho = \frac{\sqrt{r\gamma}}{r+\gamma}.
\end{equation}
Recall $A_1(r, \gamma, k)$ from \eqref{eq:def_A1_laplace} and the error function \eqref{eq:def_erf}. Then
\begin{equation}\label{eq:laplace_lbub}
	\begin{aligned}
	(\cos\Theta)^{1/2}\erf\left(\Theta\rho\sqrt{\frac{k}{\cos\Theta}}\right) 
	&\leq 
	\frac{1}{A_1(r, \gamma, k)}\int_{-\pi}^\pi \ee^{g(\theta)}\,\dd \theta 
	\\ &\leq 
	(\cos\Theta)^{-1/2}\erf\left(\Theta\rho\sqrt{k\cos\Theta}\right) + \frac{2(\pi-\Theta)\ee^{g(\Theta)}}{A_1(r, \gamma, k)}.
	\end{aligned}
\end{equation}
\end{lemma}

\begin{remark}\label{rem:rho}
Note that $\rho$ in \eqref{eq:def_rhorc} can be viewed as a function of $r/\gamma$ which is invariant when replacing $r/\gamma$ by $\gamma/r$. Since the maximum of $\rho$ is $\frac{1}{2}$ when $r = \gamma$, for $A > 0$,
\begin{equation}
	\frac{r}{\gamma} \in \left[\frac{1}{A}, A\right] \implies \rho \in \left[\frac{\sqrt{A}}{1+A}, \frac{1}{2}\right].
\end{equation}
\end{remark}

\begin{proof}

Whenever $\theta \in [-\Theta, \Theta]$, when $M(r, \gamma)$ is defined in \eqref{eq:def_Mrc},
\begin{equation}
	\cos\Theta\frac{r\gamma}{(r+\gamma)^2} \leq M(r, \gamma) \leq \frac{1}{\cos\Theta}\frac{r\gamma}{(r+\gamma)^2}.
\end{equation}
Therefore when $\theta \in [-\Theta, \Theta]$, by \eqref{eq:Mrc_ddg},
\begin{equation}\label{eq:laplace_deg2_bd}
	-k\frac{1}{\cos\Theta}\rho^2\theta^2 \leq g(\theta) - g(0) \leq -k(\cos\Theta)\rho^2\theta^2.
\end{equation}
We have the lower bound
\begin{equation}
	\begin{aligned}
	\int_{-\pi}^\pi \ee^{g(\theta)}\,\dd \theta 
	&\geq
	\int_{-\Theta}^\Theta \ee^{g(\theta)}\,\dd \theta 
	\\ &\geq
	\int_{-\Theta}^\Theta \ee^{g(0) - k\frac{1}{\cos\Theta}\rho^2 \theta^2}\,\dd \theta
	\\ &=
	\ee^{g(0)}\left(\frac{k\rho^2}{\cos\Theta}\right)^{-1/2}\int_{-\Theta\rho\sqrt{k/\cos\Theta}}^{\Theta\rho\sqrt{k/\cos\Theta}} \ee^{-\theta^2}\,\dd \theta.
	\end{aligned}
\end{equation}
The lower bound in \eqref{eq:laplace_lbub} follows from observing
\begin{equation}\label{eq:expg0_rho}
	\ee^{g(0)}\rho^{-1} = \frac{(r+\gamma)^{2k+1}}{\sqrt{r\gamma}}
\end{equation}
and dividing by $A_1(r, \gamma, k)$.

The upper bound is similar: since $g(\theta)$ is even and $g'(\theta) < 0$ for $\theta > 0$, 
\begin{equation}
	\int_{\Theta \leq |\theta| \leq \pi} \ee^{g(\theta)}\,\dd \theta \leq \ee^{g(\Theta)} \int_{\Theta \leq |\theta| \leq \pi} \,\dd \theta = 2(\pi - \Theta)\ee^{g(\Theta)}.
\end{equation}
For the integral on $[-\Theta, \Theta]$, we use \eqref{eq:laplace_deg2_bd}:
\begin{equation}
	\begin{aligned}
	\int_{-\Theta}^\Theta \ee^{g(\theta)}\,\dd \theta 
	&\leq 
	\int_{-\Theta}^\Theta \ee^{g(0) - k(\cos\Theta)\rho^2 \theta^2}\,\dd \theta 
	\\ &= 
	\ee^{g(0)}(k\rho^2\cos\Theta)^{-1/2}\int_{-\Theta\rho\sqrt{k\cos\Theta}}^{\Theta\rho\sqrt{k\cos\Theta}} \ee^{-\theta^2}\,\dd \theta.
	\end{aligned}
\end{equation}
Combining the estimates for $\Theta \leq |\theta| \leq \pi$ and for $|\theta| \leq \Theta$, using \eqref{eq:expg0_rho}, and dividing by $A_1(r, \gamma, k)$, we obtain the upper bound in \eqref{eq:laplace_lbub} and thereby complete the proof of the lemma.
\end{proof}

Our principal application of the estimates in \eqref{lem:laplace_theta} comes from making some relatively weak assumptions on $\Theta$ and $r/\gamma$.

\begin{proposition}\label{prop:laplace_app}
Let $k \in \Bbb{N}\backslash\{0\}$, let $r, \gamma > 0$, let $A > 1$, and let $\Theta \in (0, \frac{\pi}{3}]$. Recall $A_1(r, \gamma, k)$ from \eqref{eq:def_A1_laplace}, $g(\theta)$ from \eqref{eq:def_gtheta}, and define
\begin{equation}\label{eq:def_T_Theta_rho}
	T = \Theta\frac{\sqrt{A}}{1+A}.
\end{equation}
Then, if $r/\gamma \in [\frac{1}{A}, A]$,
\begin{equation}\label{eq:prop_laplace_app}
	\begin{aligned}
	(\cos\Theta)^{1/2}\left(1 - \frac{1}{T\sqrt{2\pi k}}\ee^{-\frac{1}{2}kT^2}\right) 
	& \leq 
	\frac{1}{A_1(r, \gamma, k)}\int_{-\pi}^\pi \ee^{g(\theta)}\,\dd \theta 
	\\ & \leq 
	(\cos\Theta)^{-1/2} + \frac{2}{3}\sqrt{\pi k}\ee^{-\frac{9}{\pi^2}kT^2}.
	\end{aligned}
\end{equation}
\end{proposition}

\begin{proof}
Our assumption on $r/\gamma$ implies that
\begin{equation}
	\rho \in \left[\frac{\sqrt{A}}{1+A}, \frac{1}{2}\right]
\end{equation}
(Remark \ref{rem:rho}). We estimate the tail of the error function for $z > 0$ by
\begin{equation}
	\int_z^\infty \ee^{-t^2}\,\dd t = \int_{z^2}^\infty \frac{1}{2\sqrt{s}}\ee^{-s}\,\dd s \leq \frac{1}{2z} \ee^{-z^2},
\end{equation}
from which we conclude that for $z > 0$
\begin{equation}\label{eq:erfc_bounds}
	0 < 1-\erf(z) \leq \frac{1}{z\sqrt{\pi}}\ee^{-z^2}.
\end{equation}

Recalling that $\cos\Theta \geq \frac{1}{2}$, this implies that
\begin{equation}
	1 - \frac{1}{T\sqrt{2\pi k}}\ee^{-\frac{1}{2}kT^2} \leq \erf(T\sqrt{k/2}) \leq \erf(\Theta\rho\sqrt{k\cos\Theta}) \leq \erf\left(\Theta\rho\sqrt{\frac{k}{\cos\Theta}}\right) \leq 1.
\end{equation}
The lower bound in \eqref{eq:prop_laplace_app} is then immediate from the lower bound in \eqref{eq:laplace_lbub}.

For the upper bound, it suffices to show that
\begin{equation}\label{eq:laplace_app_ub_suff}
	\frac{2(\pi - \Theta)}{A_1(r, \gamma, k)}\ee^{g(\Theta)} \leq \frac{2}{3}\sqrt{\pi k}\ee^{-\frac{9}{\pi^2}kT^2}
\end{equation}
We compute
\begin{equation}\label{eq:laplace_app_ub_term2}
	\begin{aligned}
	\frac{2(\pi - \Theta)}{A_1(r, \gamma, k)}\ee^{g(\Theta)} &= \frac{2(\pi-\Theta)\sqrt{kr\gamma}(r^2 + \gamma^2 + 2r\gamma \cos\Theta)^{k}}{\sqrt{\pi}(r+\gamma)^{2k+1}} 
	\\ &= \frac{2(\pi -\Theta)}{\sqrt{\pi}}\rho\sqrt{k}\left(1 - 2(1-\cos\Theta)\rho^2\right)^k.
	\end{aligned}
\end{equation}
Since $\Theta \in (0, \frac{\pi}{3}]$, 
\begin{equation}
	\frac{2(\pi-\Theta)}{\sqrt{\pi}} \leq \frac{4}{3}\sqrt{\pi}.
\end{equation}

We note the elementary estimates
\begin{equation}\label{eq:elem_bds_log}
	\log (1+y) \leq y, \quad \forall y \in (-1, \infty),
\end{equation}
and
\begin{equation}\label{eq:elem_bds_cos}
	\frac{1-\cos\Theta}{\Theta^2} \in \left[\frac{9}{2\pi^2}, \frac{1}{2}\right), \quad \forall \Theta \in \left(0, \frac{\pi}{3}\right]
\end{equation}
because $f(\Theta) = \frac{1-\cos\Theta}{\Theta^2}$ is decreasing on $(0, \pi)$. (This follows from the observation 
\begin{equation}
	\Theta^3 f'(\Theta) = \Theta\sin\Theta - 2(1-\cos\Theta),
\end{equation}
which vanishes to second order and has second derivative $-\Theta\sin\Theta$.)

Since $\rho \leq \frac{1}{2}$ and $\cos\Theta \geq \frac{1}{2}$, we have
\begin{equation}
	-2(1-\cos\Theta)\rho^2 \geq -\frac{1}{4}.
\end{equation}
We can therefore apply \eqref{eq:elem_bds_log} and \eqref{eq:elem_bds_cos} along with Remark \ref{rem:rho}:
\begin{equation}
	\log(1-2(1-\cos\Theta)\rho^2) \leq -\frac{9}{\pi^2}\rho^2\Theta^2 \leq -\frac{9}{\pi^2}\frac{A}{(1+A)^2}\Theta^2 = -\frac{9}{\pi^2}T^2,
\end{equation}
with $T$ defined in \eqref{eq:def_T_Theta_rho}. Using again $\rho \leq \frac{1}{2}$, we continue the computation from \eqref{eq:laplace_app_ub_term2}:
\begin{equation}
	\begin{aligned}
	\frac{2(\pi - \Theta)}{A_1(r,\gamma, k)}\ee^{g(\Theta)}
	&= \frac{2(\pi - \Theta)}{\sqrt{\pi}}\rho\sqrt{k}\exp\left(k\log(1-2(1-\cos\Theta)\rho^2)\right)
	\\ &\leq \frac{2}{3}\sqrt{\pi k}\exp\left(-\frac{9}{\pi^2}kT^2\right).
	\end{aligned}
\end{equation}
As observed in \eqref{eq:laplace_app_ub_suff}, this suffices to complete the proof of the proposition.
\end{proof}

\section{Applying Laplace's method to the integral in $r$}\label{s:laplace_r}

Next, we consider
\begin{equation}
	\int_0^\infty A_1(r, \gamma, k)r\ee^{-r^2}\,\dd r = \sqrt{\frac{\pi}{\gamma k}}\int_0^\infty \sqrt{r}\ee^{G(r)}\,\dd r,
\end{equation}
when
\begin{equation}\label{eq:G1}
	G(r) = (2k+1)\log(r+\gamma) - r^2.
\end{equation}
To apply Laplace's method, we compute the derivative
\begin{equation}
	G'(r) = \frac{2k_1}{r+\gamma}(k_1 - \gamma r - r^2)
\end{equation}
with $k_1 = k+\frac{1}{2}$, which vanishes at
\begin{equation}
	r_1 = -\frac{\gamma}{2} + \sqrt{k_1 + \frac{\gamma^2}{4}},
\end{equation}
and the second derivative
\begin{equation}\label{eq:ddG1}
	G''(r) = -2\left(\frac{k_1}{(r+\gamma)^2} + 1\right).
\end{equation}

For the exterior of the interval $[r_1-\Delta, r_1+\Delta]$, where $\Delta \geq 0$ will be determined later, we can use Lemma \ref{lem:laplace_theta_weak} to give an upper bound. This is useful even when $\Delta = 0$ because there are no restrictions on $k$ or $\gamma$.

\begin{lemma}\label{lem:tails}
Let $\Delta \geq 0$, $k \in \Bbb{N}$, and $\gamma > 0$. With $g(\theta)$ defined in \eqref{eq:def_gtheta}, $G(r)$ defined in \eqref{eq:G1}, and $r_1$ defined in \eqref{eq:def_r1},
\begin{equation}
	0 \leq \int_{r > 0, |r-r_1| \geq \Delta} r\ee^{-r^2}\int_{-\pi}^\pi \ee^{g(\theta)}\,\dd \theta \,\dd r 
	\leq
	2\pi^{3/2} \ee^{G(r_1)}(1-\erf(\Delta)).
\end{equation}
\end{lemma}

\begin{remark}\label{rem:tails_Delta0}
When $\Delta = 0$, note that the integral in Lemma \ref{lem:tails} is
\begin{equation}
	\int_0^\infty r\ee^{-r^2} \int_{-\pi}^\pi \ee^{g(\theta)}\,\dd \theta \,\dd r = \|(x+\gamma)^k\|_{\Fock}^2
\end{equation}
as in \eqref{eq:laplace_proj_integral_recall} and Remark \ref{rem:norm_fock}.
\end{remark}

\begin{proof}
The lower bound is obvious because the integrand is positive.

We begin with \eqref{eq:laplace_ub_weak} and the fact that $r \leq r+\gamma$, which gives
\begin{equation}
	\begin{aligned}
	 \int_{r > 0, |r-r_1| \geq \Delta} r\ee^{-r^2}\int_{-\pi}^\pi \ee^{g(\theta)}\,\dd \theta \,\dd r 
	 &\leq 
	 2\pi \int_{r > 0, |r-r_1| \geq \Delta} r(r+\gamma)^{2k}\ee^{-r^2}\,\dd r 
	 \\ &\leq 
	 2\pi \int_{r > 0, |r-r_1| \geq \Delta} \ee^{G(r)}\,\dd r.
	 \end{aligned}
\end{equation}
Then, from \eqref{eq:ddG1}, we see that $G''(r) \leq -2$ for all $r \geq 0$, so
\begin{equation}
	G(r) \leq G(r_1) - (r-r_1)^2
\end{equation}
for any $r > 0$. Decomposing the integral into $r < r_1 - \Delta$ and $r > r_1 + \Delta$,
\begin{equation}
	\int_{r > 0, |r-r_1| \geq \Delta} r\ee^{-r^2}\int_{-\pi}^\pi \ee^{g(\theta)}\,\dd \theta \,\dd r
	\leq 
	4\pi\int_{r_1 + \Delta}^\infty \ee^{G(r_1) - (r-r_1)^2}\,\dd r.
\end{equation}
A change of variables and the definition \eqref{eq:def_erf} of the error function completes the proof.
\end{proof}

Next, we obtain a generally applicable lower bound from \eqref{eq:laplace_lb_weak}.

\begin{corollary}\label{cor:laplace_lb_r_weak}
Let $\gamma > 0$ and $k \in \Bbb{N}\backslash \{0\}$. With $\JJ$ from \eqref{eq:def_J2_laplace}, $G$ from \eqref{eq:G1}, $r_1$ from \eqref{eq:def_r1}, and $C_0 = \sqrt{\pi}\erf(\frac{\pi}{2}) \approx 1.7258$,
\begin{equation}
	\JJ \geq C_0\sqrt{\frac{\pi}{8k}}\frac{r_1}{r_1+\gamma}\ee^{G(r_1)}.
\end{equation}
\end{corollary}

\begin{proof}
From \eqref{eq:laplace_lb_weak},
\begin{equation}\label{eq:laplace_lb_weak_r_1}
	\JJ \geq \frac{C_0}{\sqrt{k}}\int_0^\infty r(r+\gamma)^{2k}\ee^{-r^2}\,\dd r \geq \frac{C_0}{\sqrt{k}}\int_{r_1}^\infty \frac{r}{r+\gamma}(r+\gamma)^{2k+1}\ee^{-r^2}\,\dd r.
\end{equation}
When $r \geq r_1$,
\begin{equation}
	\frac{r}{r+\gamma} \geq \frac{r_1}{r_1+\gamma}.
\end{equation}
Since $\gamma > 0$,
\begin{equation}\label{eq:r1_gamma_k1}
	r_1 + \gamma = \frac{\gamma}{2} + \sqrt{k_1 + \frac{\gamma^2}{4}} \geq \sqrt{k},
\end{equation}
when $r \geq r_1$,
\begin{equation}
	G''(r) = -2\left(1 + \frac{k_1}{(r + \gamma)^2}\right) \geq -2\left(1+\frac{k_1}{(r_1+\gamma)^2}\right) \geq -4.
\end{equation}
Therefore when $r \geq r_1$,
\begin{equation}
	(r+\gamma)^{2k+1}\ee^{-r^2} = \ee^{G(r)} \geq \ee^{G(r_1) - 2(r-r_1)^2}.
\end{equation}

Inserting into \eqref{eq:laplace_lb_weak_r_1},
\begin{equation}
	\JJ \geq \frac{C_0}{\sqrt{k}}\frac{r_1}{r_1+\gamma}\ee^{G(r_1)}\int_{r_1}^\infty \ee^{-2(r-r_1)^2}\,\dd r.
\end{equation}
The integral on the right is equal to $\sqrt{\frac{\pi}{8}}$; this completes the proof of the corollary.
\end{proof}

Having established established general upper and lower bounds, we sharpen these results by applying Laplace's method. This result is designed to be used with the bounds from Proposition \ref{prop:laplace_app} and with $\delta \to \infty$ such that $\frac{\delta}{r_1}\to 0$, where $r_1$ is in \eqref{eq:def_r1}.

The integral we intend to approximate is
\begin{equation}
	\int_0^\infty A_1(r, \gamma, k)r\ee^{-r^2}\,\dd r = \int_0^\infty \sqrt{\frac{\pi r}{\gamma k}}\ee^{G(r)}\,\dd r,
\end{equation}
where $G(r)$ is defined in \eqref{eq:G1}. Heuristically, Laplace's method gives the approximation
\begin{equation}
	\int f(x)\ee^{\lambda F(x)} \, \dd x \approx f(x_1)\left(-\frac{1}{2\pi}\lambda F''(x_1)\right)^{-1/2}\ee^{\lambda F(x_1)}
\end{equation}
as $\lambda \to \infty$ for $x_1$ a nondegenerate critical point witnessing the maximum of $F$. In analogy with the coefficient $f(x_1)\left(-\frac{1}{2\pi}\lambda F''(x_1)\right)^{-1/2}$ we introduce
\begin{equation}\label{eq:def_Xdelta}
	\begin{aligned}
	X(\delta) &= \sqrt{\frac{\pi(r_1+\delta)}{\gamma k}}\left(-\frac{1}{2\pi}G''(r_1+\delta)\right)^{-1/2}
	\\ &= \pi\sqrt{\frac{r_1 + \delta}{\gamma k}}\left(1 + \frac{k_1}{(r_1 + \delta + \gamma)^2}\right)^{-1/2}.
	\end{aligned}
\end{equation}

We obtain the following refined estimate for the integral in $r$. We remark that \eqref{eq:laplace_r_B0B1} is simply a placeholder for \eqref{eq:prop_laplace_app}.

\begin{proposition}\label{prop:laplace_r}
Recall the definitions of $g(\theta)$ in \eqref{eq:def_gtheta}, $r_1$ in \eqref{eq:def_r1}, and $A_1(r, \gamma, k)$ in \eqref{eq:def_A1_laplace}, and $X$ in \eqref{eq:def_Xdelta}. Suppose that $\gamma > 0$, $k \in\Bbb{N}\backslash \{0\}$, and $\Delta \in (0, r_1)$ are such that there exist $B_0, B_1 > 0$ for which
\begin{equation}\label{eq:laplace_r_B0B1}
	B_0 \leq \frac{1}{A_1(r, \gamma, k)}\int_{-\pi}^\pi \ee^{g(\theta)}\,\dd \theta \leq B_1, \quad \forall r \in [r_1 - \Delta, r_1 + \Delta].
\end{equation}
Then, with $\JJ$ from \eqref{eq:def_J2_laplace},
\begin{equation}
	\begin{aligned}
	B_0X(-\Delta)\erf(\Delta) &\leq \ee^{-G(r_1)}\JJ
	\leq B_1X(\Delta) + 2\pi^{3/2}(1-\erf(\Delta))	\end{aligned}
\end{equation}
\end{proposition}

\begin{proof}
Recall from \eqref{eq:laplace_proj_integral_recall} and the definition \eqref{eq:def_gtheta} that
\begin{equation}
	 \JJ = \int_0^\infty r \ee^{-r^2}\int_{-\pi}^\pi \ee^{g(\theta)}\,\dd \theta\,\dd r,
\end{equation}
and using Lemma \ref{lem:tails},
\begin{equation}
	0 \leq \int_{r > 0, |r-r_1|\geq \Delta}r\ee^{-r^2}\int_{-\pi}^\pi \ee^{g(\theta)}\,\dd \theta \,\dd r \leq 2\pi^{3/2}(1-\erf(\Delta)).
\end{equation}
Note that, from the definitions \eqref{eq:def_A1_laplace} and \eqref{eq:G1},
\begin{equation}
	r\ee^{-r^2}A_1(r, \gamma, k) = \sqrt{\frac{\pi r}{\gamma k}}\ee^{G(r)}.
\end{equation}	
The hypothesis \eqref{eq:laplace_r_B0B1} implies
\begin{equation}
	\begin{aligned}
	B_0 \int_{r_1-\Delta}^{r_1+\Delta} \sqrt{\frac{\pi r}{\gamma k}}\ee^{G(r)}\,\dd r &\leq \int_{r_1-\Delta}^{r_1 + \Delta} r\ee^{-r^2}\int_{-\pi}^\pi \ee^{g(\theta)}\,\dd \theta\,\dd r 
	\\ &\leq B_1 \int_{r_1-\Delta}^{r_1+\Delta} \sqrt{\frac{\pi r}{\gamma k}}\ee^{G(r)}\,\dd r.
	\end{aligned}
\end{equation}
To complete the proof of the proposition, it is therefore sufficient to show that
\begin{equation}\label{eq:laplace_r_suff}
	\ee^{G(r_1)}X(-\Delta)\erf(\Delta) \leq \int_{r_1-\Delta}^{r_1+\Delta} \sqrt{\frac{\pi r}{\gamma k}}\ee^{G(r)}\,\dd r \leq \ee^{G(r_1)} X(\Delta).
\end{equation}

The second derivative $G''(r)$ is an increasing function for $r \in (-\gamma, \infty)$ which includes $[r_1 - \Delta, r_1+ \Delta]$ because $\gamma > 0$ and $\Delta < r_1$. Therefore on $[r_1 - \Delta, r_1+\Delta]$,
\begin{equation}
	\frac{1}{2}G''(r_1 - \Delta)(r-r_1)^2 \leq G(r) - G(r_1) \leq \frac{1}{2}G''(r_1 + \Delta)(r - r_1)^2.
\end{equation}
Therefore, with $X(\Delta)$ defined in \eqref{eq:def_Xdelta},
\begin{equation}
	\begin{aligned}
	\int_{r_1 - \Delta}^{r_1 + \Delta} &\sqrt{\frac{\pi r}{\gamma k}}\ee^{G(r)}\,\dd r  
	\leq \ee^{G(r_1)}\sqrt{\frac{\pi (r_1+\Delta)}{\gamma k}}\int_{-\Delta}^\Delta \ee^{-\frac{1}{2}G''(r_1 + \Delta)r^2}\,\dd r
	\\ &\leq \ee^{G(r_1)}\sqrt{\frac{\pi (r_1+\Delta)}{\gamma k}} \int_{-\infty}^\infty \ee^{-\frac{1}{2}G''(r_1 + \Delta)r^2}\,\dd r
	\\ &\leq \ee^{G(r_1)}\sqrt{\frac{\pi (r_1+\Delta)}{\gamma k}} \left(-\frac{1}{2}G_1''(r_1+\Delta)\right)^{-1/2}
	\\ &= \ee^{G(r_1)}X(\Delta).
	\end{aligned}
\end{equation}
This proves the right-hand inequality in \eqref{eq:laplace_r_suff}.

Since $G''(r) \leq -2$ everywhere by \eqref{eq:ddG1},
\begin{equation}
	\Delta_1 := \Delta\left(-\frac{1}{2}G''(r_1 - \Delta)\right) \geq \Delta.
\end{equation}
We insert this into a similar computation to obtain the lower bound:
\begin{equation}
	\begin{aligned}
	\int_{r_1 - \Delta}^{r_1 + \Delta} &\sqrt{\frac{\pi r}{\gamma k}}\ee^{G(r)}\,\dd r  
	\geq \ee^{G(r_1)}\sqrt{\frac{\pi (r_1-\Delta)}{\gamma k}}\int_{-\Delta}^\Delta \ee^{-\frac{1}{2}G''(r_1 - \Delta)r^2}\,\dd r
	\\ &\geq \ee^{G(r_1)}\sqrt{\frac{\pi (r_1-\Delta)}{\gamma k}} \left(-\frac{1}{2}G_1''(r_1-\Delta)\right)^{-1/2}\int_{-\Delta_1}^{\Delta_1} \ee^{-r^2}\,\dd r
	\\ &\geq \ee^{G(r_1)}X(-\Delta)\erf(\Delta).
	\end{aligned}
\end{equation}
This proves the left-hand inequality of \eqref{eq:laplace_r_suff} and therefore completes the proof of the proposition.
\end{proof}

In the proof of Theorem \ref{thm:asymp_laplace_sharp}, we need to compute $X(0)$ and control how $X$ varies. We record the necessary computations in the following lemma.

\begin{lemma}\label{lem:Xdelta}
Let $k \in \Bbb{N}$ and $\gamma > 0$, and set $k_1 = k + \frac{1}{2}$ and $u = \arsinh \frac{\gamma}{2\sqrt{k_1}}$. Then, with $X(\delta)$ from \eqref{eq:def_Xdelta},
\begin{equation}\label{eq:X0_u}
	X(0) = \pi\sqrt{\frac{k}{k_1}}(2k_1\sinh 2u)^{-1/2}.
\end{equation}
Furthermore, there exist $C_1, C_2 > 0$ such that if $\delta \in \Bbb{R}$ satisfies
\begin{equation}\label{eq:X0_cont_y_hyp}
	|\delta| \leq \frac{1}{C_1}\ee^{-u}\sqrt{k_1},
\end{equation}
then
\begin{equation}\label{eq:X0_cont}
	\left|\frac{X(\delta)}{X(0)} - 1\right| \leq C_2 |\delta| \ee^u k_1^{-1/2}.
\end{equation}
\end{lemma}

\begin{proof}
We recall the identities in \eqref{eq:arsinh_identities2}, and we compute
\begin{equation}\label{eq:X0_u_comp}
	\begin{aligned}
	X(0) &= \pi\sqrt{\frac{r_1}{\gamma k}}\left(1+\frac{k_1}{(r_1 + \gamma)^2}\right)^{-1/2} 
	\\ &= \pi\sqrt{\frac{k}{k_1}}\left(\frac{2k_1 \sinh u}{\ee^{-u}}(1+\ee^{-2u})\right)^{-1/2}
	\\ &= \pi\sqrt{\frac{k}{k_1}}(2k_1\sinh 2u)^{-1/2}.
	\end{aligned}
\end{equation}
This proves \eqref{eq:X0_u}.

Next, we compute
\begin{equation}
	\begin{aligned}
	\frac{X(\delta)}{X(0)} &= \sqrt{\frac{r_1 + \delta}{r_1}}\left(1 + \frac{k_1}{r_1 + \gamma}\right)^{1/2} \left(1 + \frac{k_1}{(r_1+\gamma)^2(1+\frac{\delta}{r_1+\gamma})^2}\right)^{-1/2}
	\\ &= \left(1 + \ee^u\delta k_1^{-1/2}\right)^{1/2}\left(\frac{1 + \ee^{-2u}(1+\delta\ee^{-u}k_1^{-1/2})}{1+\ee^{-2u}}\right)^{-1/2}
	\\ &= \left(1 + \ee^u\delta k_1^{-1/2}\right)^{1/2} \left(1 + \frac{\ee^{-2u}}{1+\ee^{-2u}}\left((1 + \delta \ee^{-u} k_1^{-1/2})^{-2}-1\right)\right).
	\end{aligned}
\end{equation}
If we set $y = \ee^{u}\delta k_1^{-1/2}$, then
\begin{equation}\label{eq:X_ratio_y}
	\frac{X(\delta)}{X(0)} = (1+y)^{1/2}\left(1 + \frac{\ee^{-2u}}{1 + \ee^{-2u}}\left((1+\ee^{-2u}y)^{-2}-1\right)\right).
\end{equation}
Since $u > 0$, $\ee^{-2u} \in (0, 1)$ and $\frac{\ee^{-2u}}{1+\ee^{-2u}} \in (0, \frac{1}{2})$. Consequently, there exist constants $C_1, C_2 > 0$ independent of $u > 0$ such that if $|y| < \frac{1}{C_1}$, then
\begin{equation}
	\left|(1+y)^{1/2}\left(1 + \frac{\ee^{-2u}}{1 + \ee^{-2u}}\left((1+\ee^{-2u}y)^{-2}-1\right)\right) - 1\right| \leq C_2 |y|.
\end{equation}
By the definition of $y$, the condition $|y| \leq \frac{1}{C_1}$ is equivalent to \eqref{eq:X0_cont_y_hyp}. In view of \eqref{eq:X_ratio_y}, this assumption is sufficient to establish \eqref{eq:X0_cont}, which completes the proof of the lemma.
\end{proof}

\section{From the integral to the spectral projection}\label{s:laplace_stirling}	
In order to pass from our asymptotics obtained from Laplace's method to the spectral projection norms for the shifted harmonic oscillator, recall that the spectral projection norm is given by
\begin{equation}
	\|\Pi_{\gamma/\sqrt{2}, k}\| = \frac{\ee^{\gamma^2/2}}{\pi k!}\JJ
\end{equation}
when $\JJ$ is defined in \eqref{eq:def_J2_laplace}. Whether we use Proposition \ref{prop:laplace_r}, Lemma \ref{lem:tails}, or Corollary \ref{cor:laplace_lb_r_weak}, we are led to consider
\begin{equation}
	\frac{\ee^{\gamma^2/2}}{k!}\ee^{G(r_1)} = \frac{1}{k!}(r_1 + \gamma)^{2k_1} \ee^{\frac{\gamma^2}{2} - r_1^2},
\end{equation}
where $r_1$ is defined in \eqref{eq:def_r1}, $k_1 = k + \frac{1}{2}$, and $G$ is defined in \eqref{eq:G1}. Setting up an application of Stirling's approximation and the change of variables in \eqref{eq:def_u_arsinh},
\begin{equation}
	\begin{aligned}
	\frac{\ee^{\gamma^2/2}}{k!}\ee^{G(r_1)}
	&=
	\frac{\sqrt{2\pi k}(k/\ee)^k}{k!} \frac{(r_1+\gamma)^{2k_1}}{\sqrt{2\pi k}(k/\ee)^k}\ee^{\frac{\gamma^2}{2} - r_1^2}
	\\ &=
	\frac{\sqrt{2\pi k}(k/\ee)^k}{k!} \frac{(r_1 + \gamma)^{2k_1}}{\sqrt{2\pi}k^{k_1}} \ee^{-r_1^2 + \frac{\gamma^2}{2} + k_1 - \frac{1}{2}}
	\\ &=
	\frac{\sqrt{2\pi k}(k/\ee)^k}{k!} \frac{1}{\sqrt{2\pi\ee}}\left(\frac{k_1}{k}\right)^{k_1} \left(\frac{r_1+\gamma}{\sqrt{k_1}}\right)^{2k_1} \ee^{\frac{\gamma^2}{2} - r_1^2 + k_1}.
	\end{aligned}
\end{equation}
Using the identities \eqref{eq:arsinh_identities2}, when $u = \arsinh\frac{\gamma}{2\sqrt{k_1}}$,
\begin{equation}\label{eq:J2_to_proj}
	\frac{\ee^{\gamma^2/2}}{k!}\ee^{G(r_1)} = \frac{\sqrt{2\pi k}(k/\ee)^k}{k!} \frac{1}{\sqrt{2\pi\ee}}\left(\frac{k_1}{k}\right)^{k_1} \ee^{k_1(2u + \sinh 2u)}.
\end{equation}

By Stirling's approximation, if $k \geq 1$, then
\begin{equation}\label{eq:Stirling_csts}
	\frac{\sqrt{2\pi}}{\ee} \leq \frac{\sqrt{2\pi k}(k/\ee)^k}{k!} \leq 1.
\end{equation}
For $k$ sufficiently large, using also that $k_1 = k(1+\frac{1}{2k})$,
\begin{equation}\label{eq:Stirling_asymp}
	\frac{1}{13k_1} \leq \frac{\sqrt{2\pi k}(k/\ee)^k}{k!} - 1 \leq \frac{1}{11k_1}.
\end{equation}

For $(k_1/k)^{k_1}$, we begin by showing that this quantity is decreasing in $k$. The derivative of the logarithm is
\begin{equation}
	\begin{aligned}
	\frac{\dd}{\dd k}\left(k_1 (\log k_1 - \log k)\right) 
	&= 
	\log \frac{k_1}{k} + k_1\left(\frac{1}{k_1} - \frac{1}{k}\right)
	\\ &= \log \left(1+\frac{1}{2k}\right) - \frac{1}{2k}
	\\ &= -\frac{1}{8k^2}\left(1+\xi\right)^{-2}
	\end{aligned}
\end{equation}
for some $\xi \in (0, \frac{1}{2k})$ by Taylor's theorem. We conclude that $\log((k_1/k)^{k_1})$ is decreasing for $k > 0$, so $(k_1/k)^{k_1}$ is a decreasing function for $k > 0$ as well. Furthermore, 
\begin{equation}
	\lim_{k \to \infty}\left(\frac{k_1}{k}\right)^{k_1} = \sqrt{\ee}.
\end{equation}
Therefore if $k \geq 1$, then
\begin{equation}\label{eq:k1k_csts}
	\sqrt{\ee} \leq \left(\frac{k_1}{k}\right)^{k_1} \leq \left(\frac{3}{2}\right)^{3/2}.
\end{equation}

To obtain error bounds for large $k$, we use $(k_1/k)^{k_1} = (1-\frac{1}{2k_1})^{-k_1}$. From Taylor's theorem, for some $\xi \in (0, \frac{1}{2k_1})$,
\begin{equation}
	\log\left(1-\frac{1}{2k_1}\right) + \frac{1}{2k_1} = -\frac{1}{2}\frac{1}{(1-\xi)^2}\frac{1}{4k_1^2}.
\end{equation}
Since $(1-\xi)^{-2} \in (1, (k_1/k)^2)$ when $\xi \in (0, \frac{1}{2k_1})$, we deduce that, for every $k \geq 1$,
\begin{equation}
	-k_1\log\left(1-\frac{1}{2k_1}\right) - \frac{1}{2} \in \left(\frac{1}{8k_1}, \frac{1}{8k_1}\left(\frac{k_1}{k}\right)^2\right).
\end{equation}
Taking the exponential, using $\ee^{x} \geq 1+x$ from below and the mean value theorem from above, there exists some $K > 0$ such that for every $k \geq K$
\begin{equation}\label{eq:k1k_asymp}
	\frac{1}{8k_1} \leq \frac{1}{\sqrt{\ee}}\left(\frac{k_1}{k}\right)^{k_1} - 1 \leq \frac{1}{7k_1}.
\end{equation}

\begin{proposition}\label{prop:laplace_stirling}
Let $k \in \Bbb{N}$ and let $\gamma > 0$. Recall $k_1 = k + \frac{1}{2}$, $G$ from \eqref{eq:G1}, $r_1$ from \eqref{eq:def_r1}, and $u$ from \eqref{eq:def_u_arsinh}. Then
\begin{equation}\label{eq:laplace_stirling_weak}
	\frac{1}{\ee} \leq \frac{\ee^{\gamma^2/2}}{k!}\ee^{G(r_1) - k_1(2u + \sinh 2u)} \leq \sqrt{\frac{27}{8\pi \ee}}.
\end{equation}
Furthermore, there exists $K > 0$ sufficiently large that, whenever $k \geq K$,
\begin{equation}\label{eq:laplace_stirling_asymp}
	\frac{1}{5k_1}
	\leq 
	\sqrt{2\pi}\ee^{-k_1(2u + \sinh 2u)}\frac{\ee^{\gamma^2/2}}{k!}\ee^{G(r_1)} - 1
	\leq
	\frac{1}{4k_1}.
\end{equation}
\end{proposition}

\begin{proof}
In the particular case $k = 0$, using \eqref{eq:arsinh_identities2} gives
\begin{equation}
	\frac{\ee^{\gamma^2/2}}{k!}\ee^{G(r_1) - k_1(2u + \sinh 2u)} = \frac{1}{k!}(r_1 + \gamma)\ee^{\frac{\gamma^2}{2} - r_1^2 - u - \frac{1}{2}\sinh 2u} = (2\ee)^{-1/2}.
\end{equation}
Since $\ee^{-1} \leq (2\ee)^{-1/2} \leq \sqrt{\frac{27}{8\pi\ee}}$, \eqref{eq:laplace_stirling_weak} is proven for $k = 0$.

For $k \geq 1$, we have shown in \eqref{eq:J2_to_proj} that 
\begin{equation}
	\begin{aligned}
	\sqrt{2\pi}\frac{\ee^{\gamma^2/2}}{k!}\ee^{G(r_1)-k_1(2u + \sinh 2u)} 
	&=
	\frac{\sqrt{2\pi k}(k/\ee)^k}{k!}\frac{1}{\sqrt{\ee}}\left(\frac{k_1}{k}\right)^{k_1}
	\\ &= (1+A)(1+B),
	\end{aligned}
\end{equation}
where for $k \geq 1$, by \eqref{eq:Stirling_csts} and \eqref{eq:k1k_csts},
\begin{equation}
	\frac{\sqrt{2\pi}}{\ee} \leq 1+A \leq 1, \quad 1 \leq 1+B \leq \frac{1}{\sqrt{\ee}}\left(\frac{3}{2}\right)^{3/2},
\end{equation}
and for $k$ sufficiently large, by \eqref{eq:Stirling_asymp} and \eqref{eq:k1k_asymp},
\begin{equation}
	\frac{1}{13k_1} \leq A \leq \frac{1}{11k_1}, \quad \frac{1}{8k_1} \leq B \leq \frac{1}{7k_1}.
\end{equation}
The former gives \eqref{eq:laplace_stirling_weak}. The latter gives \eqref{eq:laplace_stirling_asymp} for $k$ sufficiently large, because
\begin{equation}
	\frac{1}{5} < \frac{1}{8}+\frac{1}{13}, \quad \frac{1}{4} > \frac{1}{11} + \frac{1}{7}.
\end{equation}
This completes the proof of the proposition.
\end{proof}

\section{Proofs of Theorems \ref{thm:asymp_laplace_weak} and \ref{thm:asymp_laplace_sharp}}\label{s:laplace_proofs}	
		
The asymptotics obtained from Laplace's method for the integral in $\theta$ then in $r$ and an application of Stirling's formula, when combined, allow us to prove asymptotics for the spectral projections of the shifted harmonic oscillator.

\begin{proof}[Proof of Theorem \ref{thm:asymp_laplace_weak}] We use $\gamma = 2\sqrt{k_1}\sinh u$ throughout.

Using Remark \ref{rem:norm_fock} and Proposition \ref{lem:tails} (with Remark \ref{rem:tails_Delta0}) then Proposition \ref{prop:laplace_stirling},
\begin{equation}
	\begin{aligned}
	\|\Pi_{\gamma/\sqrt{2}, k}\|
	&=
	\frac{\ee^{\gamma^2/2}}{\pi k!}\|(x + \gamma)^k\|_\Fock^2
	\\ &\leq \frac{\ee^{\gamma^2/2}}{\pi k!}2\pi^{3/2}\ee^{G(r_1)}
	\\ &\leq 2\pi^{1/2}\sqrt{\frac{27}{8\pi\ee}}\ee^{k_1(2u + \sinh 2u)}
	\\ &\leq \sqrt{\frac{27}{2\ee}}\ee^{k_1(2u + \sinh 2u)}.
	\end{aligned}
\end{equation}

For the lower bound, we handle the case $k = 0$ separately. There, $k_1 = \frac{1}{2}$ and 
\begin{equation}
	\JJ = \iint \ee^{-x_1^2 - x_2^2}\,\dd x_1 \,\dd x_2 = \pi.
\end{equation}
Noting that $\gamma/\sqrt{2} = \sqrt{2k_1}\sinh u$ means that $\gamma = \sqrt{2}\sinh u$ when $k_1 = \frac{1}{2}$,
\begin{equation}
	\|\Pi_{\sinh u, 0}\| = \ee^{\gamma^2/2} = \ee^{\sinh^2 u}.
\end{equation}
Then
\begin{equation}
	\ee^{-k_1(2u+\sinh 2u)}\|\Pi_{\sinh u, 0}\| = \ee^{-u + \frac{1}{2}(1-\ee^{-2u})} \geq \ee^{-u}.
\end{equation}
Since $\frac{C_0}{\ee\sqrt{8 \pi k_1}} \approx 0.1791$, the lower bound in Theorem \ref{thm:asymp_laplace_weak} holds for $k = 0$.

If $k \geq 1$, then we apply Corollary \ref{cor:laplace_lb_r_weak} followed by Proposition \ref{prop:laplace_stirling}. Recalling that $k_1 \geq k$ and $\frac{r_1}{r_1 + \gamma} = \ee^{-2u}$ from \eqref{eq:arsinh_identities2}, this gives
\begin{equation}
	\begin{aligned}
	\frac{\ee^{\gamma^2/2}}{\pi k!}\|(x+\gamma)^k\|_\Fock^2
	&\geq 
	\frac{C_0}{\ee\sqrt{8\pi k_1}}\ee^{-2u}\frac{\ee^{\gamma^2/2}}{k!}\ee^{G(r_1)}
	\\ &\geq
	\frac{C_0}{\ee\sqrt{8\pi k_1}}\ee^{k_1(2u - \sinh 2u)-2u}.
	\end{aligned}
\end{equation}
This proves the lower bound and therefore completes the proof of Theorem \ref{thm:asymp_laplace_weak}.
\end{proof}

\begin{proof}[Proof of Theorem \ref{thm:asymp_laplace_sharp}]
To apply \eqref{eq:prop_laplace_app} in Proposition \ref{prop:laplace_app} to \eqref{eq:laplace_r_B0B1} in Proposition \ref{prop:laplace_r}, we assume that $r > 0$, $\gamma > 0$, $k \geq 1$,
\begin{equation}\label{eq:thmsharp_hyp_ThetaDelta}
	\Theta \in \left(0, \frac{\pi}{3}\right), \quad \Delta \in (0, r_1),
\end{equation}
and $A > 0$ is such that
\begin{equation}\label{eq:thmsharp_hyp_A}
	r \in [r_1 - \Delta, r_1 + \Delta] \implies \frac{\gamma}{r} \in [A^{-1}, A].
\end{equation}
Under these assumptions,
\begin{equation}\label{eq:pf_thmsharp_1}
	B_0 \frac{X(-\Delta)}{X(0)}\erf(\Delta) \leq \frac{1}{X(0)\ee^{G(r_1)}}\|(x + \gamma)^k\|_{\Fock}^2 \leq B_1 \frac{X(\Delta)}{X(0)} + \frac{1}{X(0)}2\pi^{3/2}(1-\erf(\Delta))
\end{equation}
with
\begin{equation}
	B_0 = (\cos \Theta)^{1/2}\left(1-\frac{1}{T\sqrt{2\pi k}}\ee^{-\frac{1}{2}kT^2}\right)
\end{equation}
and
\begin{equation}
	B_1 = (\cos \Theta)^{-1/2} + \frac{2}{3}\sqrt{\pi k}\ee^{-\frac{9}{\pi^2}kT^2}.
\end{equation}
Recall from Lemma \ref{lem:Xdelta} that
\begin{equation}
	\frac{X(\pm \Delta)}{X(0)} = 1 + \BigO(\Delta \ee^u k_1^{-1/2})
\end{equation}
when $\Delta \ee^u k_1^{-1/2}$ is sufficiently small.

Our goal is to establish that, when $k \geq K$ for some $K > 0$ sufficiently large,
\begin{equation}\label{eq:thmsharp_hyp_cosTheta}
	(\cos \Theta)^{-1/2}, (\cos \Theta)^{1/2} = 1 + \BigO(k_1^{-p_1}),
\end{equation}
\begin{equation}\label{eq:thmsharp_hyp_T}
	\frac{1}{T\sqrt{2\pi k}}\ee^{-\frac{1}{2}kT^2}, \sqrt{\pi k} \ee^{-\frac{9}{\pi^2}kT^2} = \BigO(k_1^{-p_1}),
\end{equation}
\begin{equation}\label{eq:thmsharp_hyp_X}
	\Delta \ee^u k^{-1/2} = \BigO(k_1^{-p_1}),
\end{equation}
and
\begin{equation}\label{eq:thmsharp_hyp_erfDelta}
	1-\erf(\Delta), \frac{1}{X(0)}(1-\erf(\Delta)) = \BigO(k_1^{-p_1}).
\end{equation}
If we accomplish this, then by \eqref{eq:pf_thmsharp_1}, we will have shown that
\begin{equation}
	\|(x + \gamma)^k\|_{\Fock}^2 = X(0)\ee^{G(r_1)}\left(1+\BigO(k_1^{-p_1})\right).
\end{equation}
Then from Theorem \ref{thm:norm_integral}, Lemma \ref{lem:Xdelta}, and Proposition \ref{prop:laplace_stirling}, we obtain
\begin{equation}\label{eq:thmsharp_end}
	\begin{aligned}
	\|\Pi_{\gamma/\sqrt{2}, k}\| 
	&= 
	\frac{\ee^{\gamma^2/2}}{\pi k!}\|(x+\gamma)^k\|_{\Fock}^2
	\\ &= 
	\frac{1}{\pi}X(0)\frac{\ee^{\gamma^2/2}}{k!}\ee^{G(r_1)}\left( 1 + \BigO(k_1^{-p_1})\right)
	\\ &= \sqrt{\frac{k}{k_1}}(2k_1 \sinh 2u)^{-1/2}\frac{1}{\sqrt{2\pi}}\ee^{k_1(2u + \sinh 2u)}\left(1+\BigO(k_1^{-1})\right)\left(1+\BigO(k_1^{-p_1})\right)
	\\ &= (4\pi k_1\sinh 2u)^{-1/2}\ee^{k_1(2u + \sinh 2u)}\left(1+\BigO(k_1^{-1})\right)\left(1+\BigO(k_1^{-p_1})\right).
	\end{aligned}
\end{equation}
So long as $p_1 < 1$, which will be the case, this suffices to prove Theorem \ref{thm:asymp_laplace_sharp}.

We turn therefore to establishing the hypotheses \eqref{eq:thmsharp_hyp_ThetaDelta}, \eqref{eq:thmsharp_hyp_A}, and \eqref{eq:thmsharp_hyp_cosTheta}--\eqref{eq:thmsharp_hyp_erfDelta}. To obtain \eqref{eq:thmsharp_hyp_cosTheta}, we simply set
\begin{equation}
	\Theta = k_1^{-p_1/2}.
\end{equation}
Consequently, $\Theta \in (0, \frac{\pi}{3})$ for $k$ sufficiently large.

Next, we suppose, for some $p_2, p_3 > 0$ to be determined below,
\begin{equation}\label{eq:thmsharp_hyp_euDelta}
	\ee^{u} \leq k_1^{p_2}, \quad \Delta = k_1^{p_3}.
\end{equation}
To establish \eqref{eq:thmsharp_hyp_X} for $k$ sufficiently large, it is necessary and sufficient to suppose that
\begin{equation}\label{eq:thmsharp_hyp_p1p2p3}
	p_1 + p_2 + p_3 < \frac{1}{2}.
\end{equation}

Under these assumptions, using Lemma \ref{lem:Xdelta},
\begin{equation}
	\frac{1}{X(0)} = \frac{1}{\pi}\sqrt{\frac{k_1}{k}}(2k_1 \sinh 2u)^{1/2} = \BigO(k_1^{\frac{1}{2}+p_2}).
\end{equation} 
Since $1-\erf(\Delta) \to 0$ more rapidly than any negative power of $\Delta$ (and therefore more rapidly than any negative power of $k_1$) by \eqref{eq:erfc_bounds}, we automatically get \eqref{eq:thmsharp_hyp_erfDelta}.

As for \eqref{eq:thmsharp_hyp_A}, since $r_1 = \sqrt{k_1}\ee^{-u}$ by \eqref{eq:arsinh_identities2}, 
\begin{equation}
	\frac{\Delta}{r_1} = \frac{k_1^{p_3}}{\sqrt{k_1}\ee^{-u}} \leq k_1^{p_2 + p_3 - \frac{1}{2}} \to 0, \quad k_1 \to \infty,
\end{equation}
by \eqref{eq:thmsharp_hyp_euDelta} and \eqref{eq:thmsharp_hyp_p1p2p3}. Therefore when $k$ is sufficiently large, $0 \leq \frac{\Delta}{r_1} \leq \frac{1}{2}$, and 
\begin{equation}
	\frac{1}{2}r_1 \leq r_1 - \Delta \leq r_1 + \Delta \leq \frac{3}{2}r_1.
\end{equation}
Since $\frac{r_1}{\gamma} = (\ee^{2u}-1)^{-1}$, substituting and takes the reciprocal gives that \eqref{eq:thmsharp_hyp_A} is implied by
\begin{equation}
	\frac{1}{A} \leq \frac{2}{3}(\ee^{2u}-1) \leq 2(\ee^{2u}-1) \leq A.
\end{equation}
In view of \eqref{eq:thmsharp_hyp_euDelta}, we set
\begin{equation}
	A = 2k_1^{2p_2}.
\end{equation}
For the left-hand inequality, we use $\ee^{2u}-1 \geq 2u$ to obtain the sufficient condition
\begin{equation}\label{eq:thmsharp_u_lb}
	\frac{3}{4A} = \frac{3}{8}k_1^{-2p_2} \leq u.
\end{equation}

Finally,
\begin{equation}
	T = \Theta \frac{\sqrt{A}}{1+A} = k_1^{-p_1/2}\frac{\sqrt{2}k_1^{p_2}}{1 + 2k_1^{2p_2}} = \frac{1}{\sqrt{2}}k_1^{-\frac{p_1}{2} - p_2}(1 + \BigO(k_1^{-p_2})).
\end{equation}
Therefore, for $k$ sufficiently large, $T \geq \frac{1}{2}k_1^{-\frac{p_1}{2} - p_2}$ and
\begin{equation}
	kT^2 \geq \frac{1}{2}k_1^{1-p_1 - 2p_2}.
\end{equation}
By \eqref{eq:thmsharp_hyp_p1p2p3}, $1-p_1 - 2p_2 > 0$, so for any $c > 0$, $\ee^{-ckT^2} \to 0$ faster than any negative power of $k_1$. Since $T^{-1} = \BigO(k_1^{\frac{p_1}{2} + p_2})$, we have \eqref{eq:thmsharp_hyp_T}.

We have established \eqref{eq:thmsharp_hyp_ThetaDelta}, \eqref{eq:thmsharp_hyp_A}, and \eqref{eq:thmsharp_hyp_cosTheta}--\eqref{eq:thmsharp_hyp_erfDelta} assuming only that $k$ is sufficiently large, that $u \leq p_2 \log k_1$ in \eqref{eq:thmsharp_hyp_euDelta}, and that $u \geq \frac{3}{8}k_1^{-2p_2}$ in \eqref{eq:thmsharp_u_lb}. In view of \eqref{eq:thmsharp_hyp_p1p2p3}, this applies to any $p_1, p_2 > 0$ for which $p_1 + p_2 < \frac{1}{2}$, since we are free to choose $p_3 = \frac{1}{2}(\frac{1}{2}-p_1-p_2)$. Following \eqref{eq:thmsharp_end}, this completes the proof of Theorem \ref{thm:asymp_laplace_sharp}.
\end{proof}

\appendix

\section{The norm of the evolution of the shifted harmonic oscillator}\label{app:SHO_norm}

We give a self-contained proof of the norm $\|\ee^{-tP_a}\|$ using the Bargmann-transform machinery introduced in Section \ref{s:integral}. The strategy used is from \cite[Thm.~3.1]{Viola_2017}, translated to the Bargmann side.

\begin{proposition}\label{prop:SHO_norm}
When $\rmRe t > 0$,
\begin{equation}
	\|\ee^{-tP_a}\| = \exp\left(a^2\frac{\cosh \rmRe t - \cos \rmIm t}{\sinh \rmRe t}\right).
\end{equation}
\end{proposition}

\begin{proof}
Recall that, with shifts as defined in \eqref{eq:def_shift},
\begin{equation}
	P_a = \shift_{(\ii a, 0)}P_0 \shift_{(\ii a, 0)}^{-1}.
\end{equation}
Let us write $\gamma = a\sqrt{2}$. Recall that
\begin{equation}
	\barg \shift_{(\ii \gamma/\sqrt{2}, 0)}\barg^* = \shift_{\frac{1}{2}(\ii \gamma, \gamma)},
\end{equation}
where the latter shift acts on the Fock space $\Fock$ defined in \eqref{eq:def_Fock}. Also recalling that $\shift_{\frac{1}{2}(\ii \gamma, -\gamma)}$ is unitary on $\Fock$, that $\barg P_0\barg^* = x\frac{\dd}{\dd x}$, and the composition rule \eqref{eq:Barg_Egorov},
\begin{equation}
	\begin{aligned}
	\left(\shift_{\frac{1}{2}(\ii \gamma, -\gamma)}\barg\right)P_{\gamma/\sqrt{2}}&\left(\shift_{\frac{1}{2}(\ii \gamma, -\gamma)}\barg\right)^{-1} 
	\\ &= \left(\shift_{\frac{1}{2}(\ii \gamma, -\gamma)}\shift_{\frac{1}{2}(\ii \gamma, \gamma)}\right)\left(x\frac{\dd}{\dd x}\right)\left(\shift_{\frac{1}{2}(\ii \gamma, -\gamma)}\shift_{\frac{1}{2}(\ii \gamma, \gamma)}\right)^{-1}
	\\ &= \shift_{(\ii \gamma, 0)}x\frac{\dd}{\dd x} \shift_{(-\ii \gamma, 0)}.
	\end{aligned}
\end{equation}
Since we are working on a space of holomorphic functions, it is easy to check that, for all $f \in \Fock$ and for all $t \in \Bbb{C}$ with $\rmRe t > 0$,
\begin{equation}
	\exp\left(-t x\frac{\dd}{\dd x}\right)f(x) = f(\ee^{-t}x).
\end{equation}
Therefore $\ee^{-tP_{\gamma/\sqrt{2}}}$ is unitarily equivalent to
\begin{equation}\label{eq:SHO_evol_Barg}
	\exp\left(-t\shift_{(\ii \gamma, 0)}x\frac{\dd}{\dd x} \shift_{(-\ii \gamma, 0)}\right): f \mapsto f(x - \ii \gamma(1-\ee^{-t}))
\end{equation}
acting on $\Fock$.

If we take two arbitrary unitary shifts acting on $\Fock$, which must be of the form $\shift_{(\lambda, -\ii \bar{\lambda})}$ and $\shift_{(\mu, -\ii \bar{\mu})}$ for $\lambda, \mu \in \Bbb{C}$, we can compute that
\begin{equation}\label{eq:SHO_FBI_SVD_ansatz}
	\shift_{(\lambda, -\ii\bar{\lambda})}\ee^{-tx\frac{\dd}{\dd x}}\shift_{(\mu, -\ii \bar{\mu})}^{-1} f(x) = \ee^{-\frac{1}{2}(|\lambda|^2 + |\mu|^2) + x(\bar{\lambda} - \ee^{-t}\bar{\mu}) + \ee^{-t}\lambda\bar{\mu}}f(\ee^{-t} x + \mu - \ee^{-t}\lambda).
\end{equation}
In order to match \eqref{eq:SHO_evol_Barg}, we solve the equations
\begin{equation}
	\begin{aligned}
	\mu - \ee^{-t}\lambda &= (1 - \ee^{-t})\ii \gamma,
	\\
	\bar{\lambda} - \ee^{-t}\bar{\mu} &= 0.
	\end{aligned}
\end{equation}
To satisfy these equations, $\mu = \ee^{\bar{t}}\lambda$ and
\begin{equation}
	\lambda = \frac{1 - \ee^{-t}}{\ee^{\bar{t}} - \ee^{-t}}\ii\gamma = \frac{\ee^t - 1}{\ee^{2\rmRe t} - 1}\ii\gamma.
\end{equation}
We remark that $\mu$ is given by the same formula as $\lambda$ except $t$ is replaced by $-t$:
\begin{equation}
	\mu = \frac{\ee^{-t} - 1}{\ee^{-2\rmRe t} - 1}.
\end{equation}

Let us write $t = t_1 + \ii t_2$ for $t_1, t_2 \in \Bbb{R}$. Since
\begin{equation}
	|\mu|^2 = |\ee^{\bar{t}}\lambda|^2 = \ee^{2t_1}|\lambda|
\end{equation}
and
\begin{equation}
	\ee^{-t}\lambda\bar{\mu} = \ee^{-t}\lambda \ee^{t}\bar{\lambda} = |\lambda|^2,
\end{equation}
we have
\begin{equation}
	\begin{aligned}
	\frac{1}{2}(|\lambda|^2 + |\mu|^2) - \ee^{-t}\lambda\bar{\mu} &= \frac{1}{2}(\ee^{2t_1}-1)|\lambda|^2 
	\\ &= \frac{1}{2(\ee^{2t_1} - 1)}|\ee^t - 1|^2 \gamma^2.
	\end{aligned}
\end{equation}
We compute
\begin{equation}
	|\ee^t - 1|^2 = (\ee^{t_1} \cos t_2 - 1)^2 + (\ee^{t_1}\sin t_2)^2 = \ee^{2t_1} - 2\ee^{t_1}\cos t_2 + 1,
\end{equation}
so
\begin{equation}\label{eq:SHO_norm_alpha}
	\frac{1}{2}(|\lambda|^2 + |\mu|^2) - \ee^{-t}\lambda\bar{\mu} = \frac{\ee^{2t_1} + 1 - 2\ee^{t_1}\cos t_2}{2(\ee^{2 t_1} - 1)}\gamma^2 = \frac{\cosh t_1 - \cos t_2}{2\sinh t_1}\gamma^2.
\end{equation}

Inserting into \eqref{eq:SHO_FBI_SVD_ansatz} and using \eqref{eq:SHO_evol_Barg},
\begin{equation}\label{eq:SHO_FBI_SVD}
	\exp\left(\frac{\cosh t_1 - \cos t_2}{2\sinh t_1}\gamma^2\right)\shift_{(\lambda, -\ii\bar{\lambda})}\ee^{-tx\frac{\dd}{\dd x}}\shift_{(\mu, -\ii \bar{\mu})}^{-1} 
	= \exp\left(-t\shift_{(\ii \gamma, 0)}x\frac{\dd}{\dd x} \shift_{(-\ii \gamma, 0)}\right),
\end{equation}
the latter operator acting on $\Fock$ being unitarily equivalent to $\exp(-tP_{\gamma/\sqrt{2}})$ acting on $L^2(\Bbb{R})$. Recall that the shifts $\shift_{(\lambda, -\ii \bar{\lambda})}$ and $\shift_{(\mu, -\ii \bar{\mu})}$ are unitary on $\Fock$ and $x\frac{\dd}{\dd x}$ is self-adjoint on $\Fock$ with spectrum $\Bbb{N}$ (either as an operator unitarily equivalent to $P_0$ or as a diagonal operator on the orthogonal basis $\{x^k\}_{k \in \Bbb{N}}$), so the operator norm of $\ee^{-tx\frac{\dd}{\dd x}}$ acting on $\Fock$ is one. Writing $\mathcal{L}(\mathcal{H})$ for the operator norm on a Hilbert space $\mathcal{H}$,
\begin{equation}
	\begin{aligned}
	\|\ee^{-tP_{\gamma/\sqrt{2}}}\|_{\mathcal{L}(L^2(\Bbb{R}))} 
	& = 
	\left\|\exp\left(-t\shift_{(\ii \gamma, 0)}x\frac{\dd}{\dd x} \shift_{(-\ii \gamma, 0)}\right)\right\|_{\mathcal{L}(\Fock)}
	\\ & =
	\exp\left(\frac{\cosh t_1 - \cos t_2}{2\sinh t_1}\gamma^2\right)\left\|\shift_{(\lambda, -\ii\bar{\lambda})}\ee^{-tx\frac{\dd}{\dd x}}\shift_{(\mu, -\ii \bar{\mu})}^{-1}\right\|_{\mathcal{L}(\Fock)}
	\\ & = \exp\left(\frac{\cosh t_1 - \cos t_2}{2\sinh t_1}\gamma^2\right).
	\end{aligned}
\end{equation}
Replacing $\gamma$ with $a\sqrt{2}$ completes the proof of the proposition.
\end{proof}

\section{The Laguerre polynomials and the spectral projection norms}

The integral formula \eqref{eq:norm_integral} in Theorem \ref{thm:norm_integral} allows us to give another proof of the relation between the Laguerre polynomials and the spectral projection norms for the shifted harmonic oscillator, first shown in \cite[Eq.~(2.20)]{Mityagin_Siegl_Viola_2016}.

\begin{proposition}\label{prop:laguerre}
For $\gamma \in \Bbb{R}$, for $\Pi_{a, k}$ the spectral projection \eqref{eq:def_proj_hermite} associated with the operator $P_a$ and the eigenvalue $k \in \Bbb{N}$, and for $L^{(0)}_k$ the Laguerre polynomials \eqref{eq:def_laguerre},
\begin{equation}
\ee^{-\gamma^2/2}\|\Pi_{\gamma/\sqrt{2}, k}\| = L^{(0)}_k(-\gamma^2).
\end{equation}
\end{proposition}

\begin{proof} When $(x_1, x_2) = (R\cos\theta, R\sin\theta)$,
\begin{equation}
	((x_1 + \gamma)^2 + x_2^2)^k = (R^2 + 2\gamma x_1 + \gamma^2)^k = \sum_{p + q + r = k}\frac{k!}{p!q!r!}R^{2p}(2 \gamma x_1)^q \gamma^{2r},
\end{equation}
where in the sum $p, q, r$ are nonnegative integers. If we integrate against $\ee^{-x_1^2-x_2^2}$, any odd power of $x_1$ integrates to zero. We therefore replace $q$ by $2q$ in the sum, which is now over $p + 2q + r = k$. When looking for the coefficient of $\gamma^{2m}$, we set
\begin{equation}\label{eq:combin_mpqr}
	m = q+r, \quad p = k-m-q, \quad r = m-q
\end{equation}
so the variables in the sum with $k, m$ fixed are determined by $q$. The requirement that $p, r \geq 0$ fixes the range of $q$ between $0$ and $\min\{m, k-m\}$. Therefore
\begin{equation}
\begin{aligned}
	\ee^{-\gamma^2/2}\|\Pi_{\gamma/\sqrt{2}, k}\| &= \frac{1}{\pi k!}\iint ((x_1+\gamma)^2 + x_2^2)^k\ee^{-x_1^2-x_2^2}\,\dd x_1 \,\dd x_2 
	\\ &= \sum_{m = 0}^k \gamma^{2m}\sum_{q = 0}^{\min\{m, k-m\}} \frac{1}{\pi p!(2q)!r!}\iint R^{2p}(2x_1)^{2q} \ee^{-x_1^2-x_2^2}\,\dd x_1 \,\dd x_2.
\end{aligned}
\end{equation}

For the integral, we switch to polar coordinates and integrate using e.g.\ \cite[Eq.~(1.1.21)]{Andrews_Askey_Roy_1999}:
\begin{equation}
\begin{aligned}
	\iint R^{2p}(2x_1)^{2q}& \ee^{-x_1^2-x_2^2}\,\dd x_1 \,\dd x_2 
	\\ &= 2^{2q}\left(\int_0^{2\pi}(\cos\theta)^{2q}\,\dd \theta\right)\left(\int_0^\infty R^{2(p+q)+1}\ee^{-R^2}\,\dd R\right)
	\\
	&= 2^{2q}\left(2\pi\frac{(2q)!}{(2^q q!)^2}\right)\left(\frac{1}{2}(p+q)!\right),
\end{aligned}
\end{equation}

So far, we have
\begin{equation}
	\ee^{-\gamma^2/2}\|\Pi_{\gamma/\sqrt{2}, k}\| = \sum_{m = 0}^k \gamma^{2m}\sum_{q = 0}^{\min\{m, k-m\}} \frac{(p+q)!}{p!r!(q!)^2}.
\end{equation}
The coefficient of $\gamma^{2m}$ is therefore, using \eqref{eq:combin_mpqr},
\begin{equation}
\begin{aligned}
	\sum_{q = 0}^{\min\{m, k-m\}} \frac{(p+q)!}{p!r!(q!)^2} &= \sum_{q = 0}^{\min\{m, k-m\}} \frac{1}{(q+r)!}\binom{q+r}{r} \binom{p+q}{q} 
	\\ &= \frac{1}{m!}\sum_{q = 0}^{\min\{m, k-m\}}\binom{m}{m-q}\binom{k-m}{q} 
	\\ &= \frac{1}{m!}\binom{k}{m},
\end{aligned}
\end{equation}
since the sum in the next-to-last line corresponds to enumerating all ways to choose $m$ objects from $k$ by choosing $q$ from the first $m$ and $m-q$ from the last $k-m$. This completes the proof of the proposition.
\end{proof}

\section{Symbols used}\label{app:symbol_index}
\begin{itemize}
\item $a$ is the parameter in the shifted harmonic oscillator \eqref{eq:def_SHO}.
\item $A_1(r, \gamma, k)$ is an approximation to the integral in $\theta$ found in the spectral projection norms; see \eqref{eq:def_A1_laplace}.
\item $A_2(\gamma, k)$ is an approximation to the integral in $r$ found in the spectral projection norms; see \eqref{eq:def_A2_laplace}.
\item $b$ is the parameter in the hypoelliptic Laplacian \eqref{eq:def_HEL}.
\item $c_0$ is the ratio between $k_1$ and $(bn)^2$ which along which $\frac{\log\|\Pi_{bn, k}\|}{\frac{1}{2}(bn)^2 + k_1}$ is maximized; see \eqref{eq:sigma_concentration}.
\item $\gamma = a\sqrt{2}$ is the rescaled parameter in the shifted harmonic oscillator, see Remark \ref{rem:norm_fock}.
\item $f_{bn, k}$ are the eigenfunctions of $L_b$; see \eqref{eq:HEL_eigenfunctions}.
\item $F(t)$ appears in the boundary of the where $\ee^{-tL_b}$ is bounded; see \eqref{eq:def_Fbdry}.
\item $g(\theta)$ is the exponent for the integral in $\theta$ to which we apply Laplace's method; see \eqref{eq:def_gtheta}.
\item $G(r)$ is the exponent for the integral in $r$ to which we apply Laplace's method; see \eqref{eq:def_Gr}.
\item $h(u)$ is an approximation to $\frac{\log\|_{\sqrt{2k_1}\sinh u, k}\|}{2k_1\cosh^2 u}$; see \eqref{eq:hu}.
\item $h_k$ are the Hermite functions; see \eqref{eq:def_hermite}.
\item $k$ is generally the eigenvalue of the shifted harmonic oscillator; see \eqref{eq:HEL_eigenfunctions}.
\item $k_1 = k + \frac{1}{2}$.
\item $L_b$ is the hypoelliptic Laplacian, \eqref{eq:def_HEL}
\item $L^{(0)}_k$ are the Laguerre polynomials, \eqref{eq:def_laguerre}.
\item $n$ is generally the energy level on the circle; see \eqref{eq:def_En}.
\item $P_a$ is the shifted harmonic oscillator, \eqref{eq:def_SHO}.
\item $\Pi_{a, k}$ is the spectral projection of the shifted harmonic oscillator, \eqref{eq:def_proj_hermite}.
\item $R(t)$ is the quantity in \eqref{eq:bdd_cond} determining boundedness of $\ee^{-tL_b}$.
\item $r_1$ is the critical point for the integral in $r$ to which we apply Laplace's method; see \eqref{eq:def_r1}.
\item $\sigma$ is the critical value beyond which the spectral decomposition for $\ee^{-tL_b}$ converges absolutely; see Definition \ref{def:sigma}.
\item $t$ is ``time'' in $\ee^{-tL_b}$, though it may be complex.
\item $\tau$ is the critical value beyond which $\ee^{-tL_b}$ is bounded whenever $\rmRe t \geq \tau$; see Definition \ref{def:tau}.
\item $u = \arsinh\frac{\gamma}{2\sqrt{k_1}}$ is a rescaling of the parameters involved in the spectral projection $\Pi_{\gamma/\sqrt{2}, k}$; see \eqref{eq:def_u_arsinh}.
\item $X(\delta)$ is an approximation to the non-exponential part of Laplace's method in the integral in $r$; see \eqref{eq:def_Xdelta}.
\end{itemize}
\bibliographystyle{plain}
\bibliography{SHO_HL}

\end{document}